\documentclass[12pt]{article}
 \usepackage{graphicx,amsmath,amsfonts,amssymb}
\usepackage[dvips]{color}
\usepackage{ulem}
\usepackage{cases}

\definecolor{brown}{cmyk}{0,0.83,1,0.6}

 \newcommand{\vc}[1]{{\boldsymbol #1}}
 \newcommand{\vcn}[1]{{\bf #1}}
 
 \newcommand{\sr}[1]{{\mathcal #1}}
 \newcommand{\dd}[1]{\mathbb{#1}}

 \newcommand{\eqn}[1]{(\ref{eqn:#1})}
 \newcommand{\lem}[1]{Lemma~\ref{lem:#1}}
 \newcommand{\cor}[1]{Corollary~\ref{cor:#1}}
 \newcommand{\thr}[1]{Theorem~\ref{thr:#1}}
 
 \newcommand{\pro}[1]{Proposition~\ref{pro:#1}}

 \newcommand{\rem}[1]{Remark~\ref{rem:#1}}
 \newcommand{\fig}[1]{Figure~\ref{fig:#1}}
\newcommand{\app}[1]{Appendix~\ref{app:#1}}
\newcommand{\sectn}[1]{Section~\ref{sec:#1}}

\newcommand{\lemt}[1]{\ref{lem:#1}}

\newcommand{\appt}[1]{\ref{app:#1}}

 \newcommand{\ol}{\overline}
 
 \newcommand{\pend}{\hfill \thicklines \framebox(6.6,6.6)[l]{}}
 \newenvironment{proof}{\noindent {\sc  Proof.} \rm}{\pend}
 \newenvironment{proof*}[1]{\noindent {\sc  #1} \rm}{\pend}

 \newtheorem{theorem}{Theorem}[section]
 \newtheorem{lemma}{Lemma}[section]
 
 \newtheorem{proposition}{Proposition}[section]
 
 \newtheorem{remark}{Remark}[section]

 \newtheorem{corollary}{Corollary}[section]

 \newcommand{\setnewcounter} {
 \setcounter{subsection}{0}
 \setcounter{equation}{0}
 \setcounter{conjecture}{0}
 \setcounter{assumption}{0}
 \setcounter{question}{0}
 \setcounter{definition}{0}
 \setcounter{theorem}{0}
 \setcounter{corollary}{0}
 \setcounter{lemma}{0}
 \setcounter{proposition}{0}
 \setcounter{remark}{0}
}

\newenvironment{mylist}[1]{\begin{list}{}
{\setlength{\itemindent}{#1mm}}
{\setlength{\itemsep}{0ex plus 0.2ex}}
{\setlength{\parsep}{0.5ex plus 0.2ex}}
{\setlength{\labelwidth}{10mm}}
}{\end{list}}

\newcommand{\dequal}{\stackrel {\rm d}{=}}

 \title{\bf Join the shortest queue among $k$ parallel queues: tail asymptotics of its stationary distribution}
\author{Masahiro Kobayashi$^{a}$, Yutaka Sakuma$^{b}$ and Masakiyo Miyazawa$^{a}$
\\ {\small $^{a}$Department of Information Sciences, Tokyo University of Science}
\\  {\small $^{b}$Hiroshima National College of Maritime Technology}
}
\date{Revised, February 7, 2013}

\begin{document}
\maketitle

\begin{abstract}
We are concerned with an $M/M$-type join the shortest queue ($M/M$-JSQ for short) with $k$ parallel queues for an arbitrary positive integer $k$,  where the servers may be heterogeneous. We are interested in the tail asymptotic of the stationary distribution of this queueing model, provided the system is stable. We prove that this asymptotic for the minimum queue length is exactly geometric, and its decay rate is the $k$-th power of the traffic intensity of the corresponding $k$ server queues with a single waiting line. For this, we use two formulations, a quasi-birth-and-death (QBD for short) process and a reflecting random walk on the boundary of the $k+1$-dimensional orthant. The QBD process is typically used in the literature for studying the JSQ with $2$ parallel queues, but the random walk also plays a key roll in our arguments, which enables us to use the existing results on tail asymptotics for the QBD process.
\end{abstract}

{\bf Keywords: }
Join the shortest queue,  heterogeneous servers, stationary distribution, \\
\indent exactly geometric asymptotics, quasi-birth-and-death process, reflecting random walk.

\section{Introduction}
\label{sec:introduction}
We consider a parallel queueing model in which customers join the shortest queue.
If there are more than one queues whose lengths are shortest, then we assume tie break with equal probabilities.
We denote the number of queues by $k$ $(\ge2)$.
It is assumed that customers arrive according to a Poisson process,  and  each queue has a single server, and its service times are {\it i.i.d} with an exponential distribution.
Here, those servers may have different mean service times, that is, they may be heterogeneous.
We refer to this queueing model as an $M/M$-type join the shortest queue  ($M/M$-JSQ for short).

We are interested in the stationary distribution of the queue length for the $M/M$-JSQ.
However, its analytic derivation is known to be hard, and theoretical interests have been directed to the tail asymptotic of the stationary distribution.
For $k=2$,
this problem has been well studied.
Kingman studied the $M/M$-JSQ with $2$ parallel queues having homogeneous servers.
He proved that  the stationary distribution of the minimum queue length  has an exactly geometric asymptotics and its decay rate is equal to the square of traffic intensity of the corresponding $M/M/2$ queue with a single waiting line, where exactly geometric asymptotics means that the tail probability is asymptotically proportional to a geometric function.
Many researchers obtained similar geometric asymptotics for more general models with two parallel queues. The case of heterogeneous servers is studied in \cite{TFM 2001}.
Foley and McDonald \cite{FM 2001} considered the generalized shortest queue which has a Poisson stream dedicated to each queue in addition to a Poisson stream which chooses the shortest queue. They obtained the stability condition and exactly geometric asymptotic under some extra conditions.
This model was further studied in \cite{HMZ 2007,Miyazawa two side}.
In particular, Miyazawa \cite{Miyazawa two side} described the generalized shortest queue by a two sided double quasi-birth-and-death (QBD) process, and derived the tail decay rates without any extra condition.
 Sakuma \cite{Sakuma 2010b} considered two parallel queues with a common phase type service time distribution and a Markov modulated arrivals, and derived exactly geometric asymptotic under a certain condition. 

All those studies for JSQ assume two parallel queues. We consider the case where there are more than two parallel queues, that is, $k$ queues with $k \ge 3$. 
Puhalskii and Vladimirov \cite{Puhalskii 2007} studied the tail asymptotics for much more general $k$ parallel queues in which there are multiple classes of customers who can only choose the shortest queue among queues assigned to them.  They derived the large deviations principle for this generalized model. However,  they do not provide any explicit asymptotics even for $k=2$. Sakuma \cite{Sakuma 2010a} also studied $k$ parallel queues with JSQ discipline, but this model is allowed to have jockeying when the maximum difference among queues is greater than a given threshold level. Due to this jockeying assumption, the problem is reduced to one dimensional queue, and the standard technique can be applied to get tail asymptotics.

For the $M/M$-JSQ with $k$ parallel queues, it is easy to guess that the tail decay rate of the stationary distribution for the minimum queue length is the $k$-th power of the traffic intensity of the corresponding $k$ server queue with a single waiting line because all the queues should be balanced. However, there has been no proof 
for $k \ge 3$ as far as we know. This may be because there is no satisfactory method for tail asymptotics on the stationary distribution for more than two correlated queues.

The aim of this paper is to prove this conjecture. This will be done by obtaining exactly geometric asymptotic. For this, we employ two formulations. We first consider the exact geometric asymptotics using a discrete time QBD process, which has two components, level and background. The level is one dimensional, and represents a characteristics of interest for tail asymptotics. The background state has all the information for the process to be Markov. We define the level by the minimum queue length for the $M/M$-JSQ with $k$ parallel queues while background state is a set of differences between the queues lengths and their minimum. Then, we have the QBD process, and the stationary distribution is represented by the so called matrix geometric solution. 
This solution can be used to derive the exactly geometric asymptotic. 
However, we need to verify certain conditions for this derivation which are not always easy to verify.
In particular, one of them is involved with the tail asymptotic of the marginal distribution at level zero, which is generally unknown.

To overcome this difficulty, we use another formulation. We take the same state space as that of the QBD process. Thus, each state is a $k+1$ dimensional vector at least one of whose entries vanish. 
We consider this process as a reflecting random walk on the boundary of the $k+1$ dimensional nonnegative integer orthant. This boundary is composed of $2(2^{k} - 1)$ faces depending on which entries vanish.  
We refer to this process as a reflecting random walk for the shortest queue. This random walk provides us a different tool for solving the tail asymptotic problem. For this, we use moment generating functions for describing the stationary equation instead of the matrix geometric solution. Of course, it is very hard to analytically derive the moment generating function of the multidimensional stationary distribution. Instead of doing so, we only consider its convergence domain, similarly to our recent papers \cite{KM 2010, Miyazawa 2010}.  We can not find the whole domain, but can get sufficient information to apply the tail asymptotic result of the QBD process.

The rest of this paper is organized as follows. In \sectn{JSQ}, we formally introduce a Markov chain for the shortest queue and formulate it in two ways, the QBD process and reflecting random walk. We then present a main result, the exact geometric asymptotic for the $M/M$-JSQ with $k$ parallel queues (\thr{JSQ main result}). As a corollary of this result, we also derive a rough asymptotic for the marginal distribution of the minimum queue length (\cor{JSQ rough marginal}). In \sectn{Geometric}, we prove the main result using one proposition and five lemmas. The first two lemmas are on the QBD process, and proved in the appendix. The last lemma plays a key roll for our proof. It  is proved in \sectn{Proof of Domain}, using further lemmas.
We give some concluding remarks in  \sectn{concluding remark}.

\section{Modeling and exactly geometric asymptotics}
\setnewcounter
\label{sec:JSQ}

	We consider a queueing model with $k$ parallel single server queues, where each waiting line has infinite capacity, and we index those $k$ queues as $1,2, \ldots, k$. Customers arrive according to a Poisson process with rate $\lambda$ and join the shortest queue, where ties are broken with equal probabilities when there are more than one shortest queues. At the $i$-th queue,  the customers are served according to first-come first-served discipline, and their service times  are independent and exponentially distributed with mean  $\mu_{i}^{-1}$. Thus, the severs may not be homogeneous.
	This queueing model is referred to as an $M/M$-type join the shortest queue ($M/M$-JSQ) with $k$ parallel queues (see also \fig{M/M/c}).
\begin{figure}[htbp]
 	\centering
	\includegraphics[height=0.17\textheight]{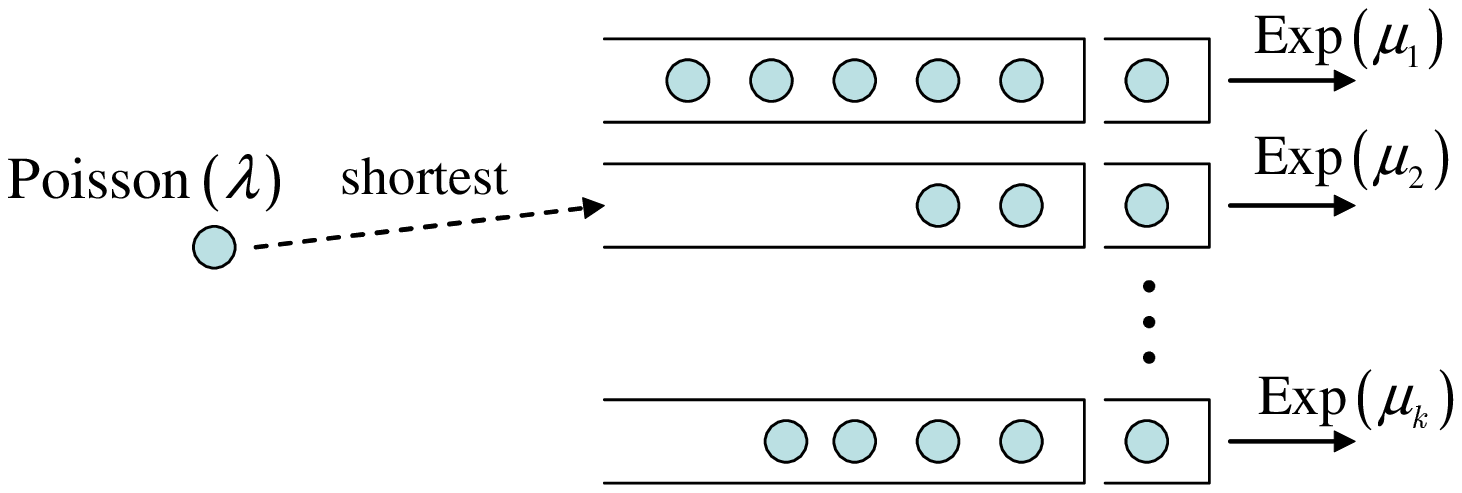}
	\caption{$M/M$-JSQ with $k$ parallel queues}
	\label{fig:M/M/c}
\end{figure}

 We denote the index set of queues by $J$, that is,
\begin{eqnarray*}
  J = \{ 1,2, \ldots, k \}.
\end{eqnarray*}
	For each $t$ and $i \in J$, let $L_{i}(t)$ be the number of customers in queue $i$ including a customer being served, and let
\begin{eqnarray*}
  \vc{L}(t) = (L_{1}(t), L_{2}(t), \ldots, L_{k}(t)).
\end{eqnarray*}
  It is easy to see that $\{\vc{L}(t) ; t \in \dd{R}_{+}\}$ is a continuous time Markov chain with state space $\dd{Z}_{+}^{k}$, where $\dd{R}_{+}$ and $\dd{Z}_{+}$ are the sets of all nonnegative real numbers and integers, respectively. We denote the traffic intensity of this queueing model by 
\begin{eqnarray*}
  \rho = \frac {\lambda}{\sum_{i=1}^{k} \mu_{i}},
\end{eqnarray*}
 and assume that
\begin{eqnarray}
\label{eqn:Stability condition}
	\rho < 1,
\end{eqnarray}
 which is known to be the stability condition (see, e.g., \cite{FM 2001}).
	Since the total transition rate from each state is bounded  by  $\lambda + \sum_{i=1}^{k} \mu_{i}$,
 we can construct a discrete time Markov chain which has the same stationary distribution as that of $\{\vc{L}(t)\}$ by uniformization.
 We normalize $\lambda + \sum_{i=1}^{k} \mu_{i}$ without loss of generality as
\begin{eqnarray} 
\label{eqn:Uniform}
	 \lambda + \sum_{i=1}^{k} \mu_{i} = 1.
\end{eqnarray}
	We denote this discrete time Markov chain by $\{\vc{L}_{\ell} ; \ell = 0,1,\ldots\}$, where
\begin{eqnarray}
\label{eqn:Discrete time process}
	\vc{L}_{\ell} = (L_{\ell 1}, L_{\ell 2}, \ldots, L_{\ell k}).
\end{eqnarray}
In this paper, we refer to this process as an original queue length process.

	In the rest of this paper, we consider this discrete time process.  The state transitions of $\vc{L}_{\ell}$ are a bit complicated because they depend on how $L_{\ell i}$'s are ordered. Thus, we describe it in a slightly different way. Let
\begin{eqnarray*}
	M_{\ell} = \min_{i \in J} L_{\ell i}, \quad
	\vc{Y}_{\ell} = (Y_{\ell 1}, Y_{\ell 2}, \ldots, Y_{\ell k}), \qquad \ell=0, 1, \ldots,
\end{eqnarray*}
  where $Y_{\ell i} = L_{\ell i} - M_{\ell}$ for $i \in J$. Let $\{\vc{Z}_{\ell}; \ell=0, 1, \ldots\}$ be a pair of $M_{\ell}$ and $\vc{Y}_{\ell}$, that is,
\begin{eqnarray*}
  \vc{Z}_{\ell} = (M_{\ell}, \vc{Y}_{\ell}),
\end{eqnarray*}
  which is just another expression for $\{\vc{L}_{\ell}\}$. Obviously,  $\{\vc{Z}_{\ell};\ell=0,1,2,\ldots\}$ is a discrete time Markov chain which takes values in $\mathbb{Z}_{+}^{k+1}$.
  
  The process $\{\vc{Z}_{\ell}\}$ is convenient because of two reasons. First, it can be considered as a quasi-birth-and-death process, QBD for short, if we view $M_{\ell}$ as level and $\vc{Y}_{\ell}$ as background state. Secondly, it can be considered as a reflecting random walk. This simplifies our arguments while keeping accurate mathematical expressions in the boundary faces of $k+1$-dimensional nonnegative integer orthant. 
  
  Let us describe the process $\{\vc{Z}_{\ell}\}$ as a $k+1$-dimensional reflecting random walk. We first partition its state space. 
Let ${\sr K} \equiv \{U \subset J; U \neq \emptyset\}$. For $U \in {\sr K}$, we define the following boundary faces.
\begin{eqnarray*}
{\sr S}_{+U} &=& \{\vc{u} = (u_{0},u_{1},u_{2},\ldots u_{k}) \in \mathbb{Z}_{+}^{k+1}; u_{0} \ge 1, u_{i} \ge 1, i \in J \setminus U, u_{j} = 0, j \in U \},\\
{\sr S}_{0U} &=& \{\vc{u} = (u_{0},u_{1},u_{2},\ldots u_{k}) \in \mathbb{Z}_{+}^{k+1}; u_{0} = 0, u_{i} \ge 1, i \in J \setminus U, u_{j} = 0, j \in U\}.
\end{eqnarray*}
Denote the boundary of $\mathbb{Z}_{+}^{k+1}$ by $\partial \mathbb{Z}_{+}^{k+1}$, that is, 
\begin{eqnarray*}
&&\partial \mathbb{Z}_{+}^{k+1} = \{\vc{u} = (u_{0},u_{1},\ldots,u_{k}) \in \mathbb{Z}_{+}^{k+1}; \exists i \in \{0\} \cup J, u_{i} = 0\}.
\end{eqnarray*}
Then, $\partial \mathbb{Z}^{k+1} \supseteq \cup_{U \in {\sr K}} (S_{+U} \cup S_{0U})$, on which $\{\vc{Z}_{\ell}\}$ stays.
Thus, the state space of $\vc{Z}_{\ell}$ is given by 
\begin{eqnarray*}
\displaystyle{ {\sr S} = \cup_{U \in {\sr K}} ({\sr S}_{+U} \cup {\sr S}_{0U})}.
\end{eqnarray*}
Note that $\vc{Z}_{\ell} \in {\sr S}_{+U}$ if $U$ is the set of indices of the shortest  queues and $M_{\ell} \ge 1$. Similarly, $\vc{Z}_{\ell} \in {\sr S}_{0U}$ implies that $U$ is the set of indices of the shortest  queues and $M_{\ell} =0$.

  Let us consider the distribution of increment $\vc{Z}_{\ell+1} - \vc{Z}_{\ell}$. It only depends on the boundary face to which $\vc{Z}_{\ell}$ belongs.   
To give its distribution, we use the following notations. For $i \in J$, let $\vcn{e}_{i}$ be the $k$-dimensional row vector whose $i$-th entry is unit and the other entries vanish (e.g., $\vcn{e}_{1} = (1,0,\ldots,0)$), and let $\vc{1}$ be the $k$-dimensional row vector whose all entries are units, i.e., $\vc{1} = (1,1, \ldots, 1)$. We denote the number of elements of set $U$ by $|U|$.
For each $U \in {\sr K}$, we define the random vector $\vc{X}^{(+U)}$ taking value in $(j, \vc{v}) \in \{0,-1,1\}^{k+1}$ as follows, for $|U| = 1$,
\begin{eqnarray}
\label{eqn:Dist of vc{X}(U)1}
\dd{P}(\vc{X}^{(+U)} = (j,\vc{v}) ) = \left \{
\begin{array}{ll}
\lambda,                        & (j, \vc{v}) = (1, - \vc{1} + \vcn{e}_{i}) , i \in U, \\ 
\mu_{i},                            & (j, \vc{v}) = (0, - \vcn{e}_{i}), i \in J \setminus U\\
                            & \mbox{or } (j, \vc{v}) = (-1, \vc{1} - \vcn{e}_{i}), i \in U,\\
0,                              & \mbox{otherwise},
\end{array}
\right.
\end{eqnarray}
and for $|U| \ge 2$,
\begin{eqnarray}
\label{eqn:Dist of vc{X}(U)2}
\dd{P}(\vc{X}^{(+U)} = (j,\vc{v}) ) = \left \{
\begin{array}{ll}
\frac{1}{|U|} \lambda, & (j,\vc{v}) = (0, \vcn{e}_{i}), i \in U,\\
\mu_{i},               & (j, \vc{v}) = (0, - \vcn{e}_{i}), i \in J \setminus U\\
		       & \mbox{or }(j, \vc{v}) = (-1, \vc{1} - \vcn{e}_{i}), i \in U,\\
0,                     \quad & \mbox{otherwise}.
\end{array}
\right.
\end{eqnarray}
Similarly, let $\vc{X}^{(0U)}$ be the random vector such that,  for $|U| = 1$ 
\begin{eqnarray}
\label{eqn:Dist of vc{X}(0,U)1}
\dd{P}(\vc{X}^{(0U)} = (j,\vc{v}) ) = \left \{
\begin{array}{ll}
\lambda,  & (j, \vc{v}) = (1, - \vc{1} + \vcn{e}_{i}), i \in U,\\ 
\mu_{i},  & (j, \vc{v}) = (0, - \vcn{e}_{i}), i \in J \setminus U\\
          & \mbox{or }(j, \vc{v}) = (0, \vc{0}),i \in U,\\
0,        \quad & \mbox{otherwise},
\end{array}
\right.
\end{eqnarray}
and for $|U| \ge 2$,
\begin{eqnarray}
\label{eqn:Dist of vc{X}(0,U)2}
\dd{P}(\vc{X}^{(0U)} = (j,\vc{v}) ) = \left \{
\begin{array}{ll}
\frac{1}{|U|} \lambda, & (j,\vc{v}) = (0, \vcn{e}_{i}), i \in U,\\
\mu_{i},               & (j, \vc{v}) = (0, - \vcn{e}_{i}), i \in J \setminus U\\
                       & \mbox{or }(j, \vc{v}) = (0, \vc{0}),i \in U,\\
0,                     & \mbox{otherwise},
\end{array}
\right.
\end{eqnarray}
where $\vc{0}$ is a null vector.
  Note that $\vc{X}^{(0U)}$ represents the increment when the queues with indices in $U$ are empty.
  Then, $\{\vc{Z}_{\ell}\}$ can be obtained as
\begin{eqnarray}
\label{eqn:random walk 1}
 \vc{Z}_{\ell+1} = \vc{Z}_{\ell} + \!\! \sum_{U \in {\sr K}} \!\! \left( \vc{X}^{(+U)}_{\ell} 1(\vc{Z}_{\ell} \in {\sr S}_{+U})
                                                   + \vc{X}^{(0U)}_{\ell} 1(\vc{Z}_{\ell} \in {\sr S}_{0U}) \right), 
\end{eqnarray}
  where $\vc{X}^{(+U)}_{\ell}$ and $\vc{X}^{(0U)}_{\ell}$ are independent copies of $\vc{X}^{(+U)}$ and $\vc{X}^{(0U)}$, respectively, and $1(\cdot)$ is the indicator function. This $\{\vc{Z}_{\ell}\}$ is referred to as a reflecting random walk for the JSQ.
  
  By the stability condition \eqn{Stability condition}, $\{\vc{Z}_{\ell}\}$ has a stationary distribution. Denote a random vector subject to this distribution by $\vc{Z} \equiv (M,\vc{Y})$. Then, \eqn{random walk 1} yields
\begin{eqnarray}
\label{eqn:stationary 1}
 \vc{Z} \dequal \vc{Z} + \sum_{U \in {\sr K}} \left( \vc{X}^{(+U)} 1(\vc{Z} \in {\sr S}_{+U})
                                                   + \vc{X}^{(0U)} 1(\vc{Z} \in {\sr S}_{0U}) \right),
\end{eqnarray}
 where $\dequal$ stands for equality in distribution, and $\vc{X}^{(+U)}$ and $\vc{X}^{(0U)}$ are assumed to be independent of $\vc{Z}$.
The stationary equation \eqn{stationary 1} plays a key roll in our arguments. 

  Let
\begin{eqnarray*}
   \sr{H} = \{ \vc{h} = (h_{1},h_{2},\ldots, h_{k}) \in \mathbb{Z}_{+}^{k} ;  \exists i \in J, h_{i} = 0\}.
\end{eqnarray*}
Obviously, ${\sr H}$ is a state space for $\vc{Y}$ and ${\sr S} = \mathbb{Z}_{+} \times {\sr H}$.
We are ready to present our main result of this paper which will be proved in \sectn{Geometric} using results in \sectn{Proof of Domain} and Appendices.
 
\begin{theorem} {\rm
\label{thr:JSQ main result}
	For the $M/M$-JSQ with $k$ parallel queues satisfying the stability condition \eqn{Stability condition}, we have for each $\vc{h} \in {\sr H}$,
\begin{eqnarray}
\label{eqn:exact asymptotic JSQ}
  \lim_{n \to \infty} \rho^{-kn} \dd{P}(M = n, \vc{Y} = \vc{h}) = c_{\vc{h}},
\end{eqnarray}
 where $c_{\vc{h}}$ is a positive constant.
}
\end{theorem}

 For two parallel queues with homogeneous severs, this theorem was firstly obtained by Kingman \cite{Kingman 1961} by using analytic functions. It is also known for two parallel queues with heterogeneous servers (e.g, see \cite{TFM 2001}). Similar results were obtained for two parallel queues under more general setting (e.g., see \cite{FM 2001,HMZ 2007,Miyazawa two side,Sakuma 2010a,TFM 2001} and references in those papers). Many of them use the QBD processes and their limiting behaviors (e.g., see \cite{MZ 2004}).
\begin{corollary} {\rm
\label{cor:JSQ rough marginal}
Under the same assumptions of \thr{JSQ main result},
\begin{eqnarray}
\label{eqn:rough asymptotic JSQ}
  \lim_{n \to \infty} \frac 1n \log \dd{P}(M = n) = \log \rho^{k}.
\end{eqnarray}
}\end{corollary}

This corollary is not immediate from \thr{JSQ main result} because \eqn{exact asymptotic JSQ} only implies that the limit infimum of $\frac 1n \log \dd{P}(M = n)$ is lower bounded by $\log \rho^{k}$. We will prove it in \app{JSQ rough marginal}.

\section{Proof of \thr{JSQ main result}}
\setnewcounter
\label{sec:Geometric}
For the proof of \thr{JSQ main result}, we use two formulations, the QBD process and the reflecting random walk for the JSQ. 
We first discuss basic results on the exact asymptotic for the QBD process (see \pro{sufficient conditions}).
We next consider the convergence domain for the moment generating function of stationary distribution. For this, we prepare some lemmas. 
The last lemma among them has a key roll in our arguments, which will be proved in \sectn{Proof of Domain}.
Finally, we prove \thr{JSQ main result} in \sectn{Proof of Theorem}.

\subsection{QBD process and sufficient conditions for geometric tail decay}
\label{sec:QBD}

We first present the tail asymptotic result for the QBD process known in the literature \cite{HMZ 2007, MZ 2004}. 
For this, we use some matrices.
For $i=0,\pm 1$ and $\ell = 0,1,2, \ldots$, define infinite dimensional matrices $B_{0}$, $A_{i}$ as
\begin{eqnarray*}
\begin{array}{llll}
&&[B_{0}]_{\vc{h},\vc{h}'} = \dd{P}(\vc{Z}_{\ell+1} = (0, \vc{h}') | \vc{Z}_{\ell} = (0,\vc{h})), &\vc{h},\vc{h}' \in {\sr H},\\
&&[A_{i}]_{\vc{h},\vc{h}'} = \dd{P}(\vc{Z}_{\ell+1} = (n + i, \vc{h}') | \vc{Z}_{\ell} = (n,\vc{h})),  &\vc{h},\vc{h}' \in {\sr H}, n \ge 1. 
\end{array}
\end{eqnarray*}
 Then, the QBD process $\{\vc{Z}_{\ell}\}$ has the following transition probability matrix. 
\begin{eqnarray*}
P  =  \left(
\begin{array}{ccccc}
B_{0} & A_{+1}  &            &            &       \\
A_{-1}& A_{0}  & A_{+1} &            &       \\
           & A_{-1} & A_{0} & A_{+1} &       \\
           &             & \ddots     & \ddots     & \ddots\\
\end{array}
\right).
\label{eq:Q}
\end{eqnarray*}
  We assume the stability condition \eqn{Stability condition}, and therefore the stationary distribution exists. We denote it by row vector $\vc{\pi}$.

  For the QBD process, it is important to distinguish level from background state. For this, we partition the stationary vector $\vc{\pi}$ as $(\vc{\pi}_{0},\vc{\pi}_{1},\ldots)$ according to level. That is,
\begin{eqnarray*}
	[\vc{\pi}_{n}]_{\vc{h}} = \dd{P}(M = n, \vc{Y} = \vc{h}),
\end{eqnarray*}
 for $n \in \dd{Z}_{+}$ and $\vc{h} \in \sr{H}$, where $[\vc{\pi}_{n}]_{\vc{h}}$ denotes the $\vc{h}$-th entry
 of $\vc{\pi}_{n}$.
	As is well known (e.g., \cite{LR 1999} and \cite{NEUTS}), the stationary distribution is known to have the following matrix geometric form:
\begin{eqnarray}
\label{eqn:geometric form}
	\vc{\pi}_{n} = \vc{\pi}_{0} R^{n}, \qquad n \ge 1,
\end{eqnarray}
 where $R$ is the minimal nonnegative solution of the following equation:
\begin{eqnarray}
\label{eqn:R-M}
	R = A_{+1} + RA_{0} + R^{2} A_{-1}.
\end{eqnarray}

  When the size of $R$ is finite, we can see that the tail decay rate of \eqn{geometric form} is obtained as the maximal eigenvalue of $R$. Otherwise, this is not always true. Thus, we need certain extra conditions here. Such conditions were firstly obtained in \cite{TFM 2001}, and refined and generalized in \cite{MZ 2004}. The following result for the geometric tail decay are the specialization of the results in \cite{MZ 2004} to the QBD process (see Theorem $2.1$ of \cite{HMZ 2007}).
\begin{proposition}
\label{pro:sufficient conditions}
{\rm 
	Assume that $A \equiv A_{-1} + A_{0} + A_{+1}$ is irreducible and aperiodic, and that the Markov additive process 
 generated by $\{ A_{\ell}; \ell = 0, \pm 1 \}$ is 1-arithmetic.
	If there exist $\alpha > 1$ and positive vectors $\vc{x}$ and $\vc{y}$ such that
\begin{eqnarray}
\label{eqn:left A-vector}
&\vc{x} A_{*}(\alpha) = \vc{x}, \quad A_{*}(\alpha) \vc{y} = \vc{y},& \\
\label{eqn:positive recurrent}
&\vc{x} \vc{y} < \infty, &
\end{eqnarray}
 where $A_{*}(z) = z^{-1} A_{-1} + A_{0} + z A_{+1}$ for $z \neq 0$,
 then $R$ has left and right eigenvectors $\vc{x}$ and $\vc{r} \equiv (I - A_{0} - R A_{-1} - \alpha^{-1} A_{-1}) \vc{y}$,
 respectively, with eigenvalue $\alpha^{-1}$.
	Furthermore, if
\begin{eqnarray}
\label{eqn:pi0 y}
	\vc{\pi}_{0} \vc{y} < \infty,
\end{eqnarray}
 then we have the following geometric tail asymptotics for the stationary distribution:
\begin{eqnarray}
\label{eqn:Geometric tail asymptotics}
\lim_{n \to \infty} \alpha^{n} \vc{\pi}_{n} = \frac{\vc{\pi}_{0} \vc{r}}{\vc{x}\vc{r}} \vc{x}.
\end{eqnarray}
}
\end{proposition}
\begin{remark}
{\rm
	Both irreducibility and aperiodicity of $A$ are easy to verify for our queueing model.
	Furthermore, the 1-arithmetic property is directly verified, that is, for each $\vc{h} \in \sr{H}$,
 the greatest common divisor of
\begin{eqnarray*}
	&& \{ \ell_{1} + \ell_{2} + \cdots + \ell_{i} ; [A_{\ell_{1}}]_{\vc{h}, \vc{h}_{1}} [A_{\ell_{2}}]_{\vc{h}_{1}, \vc{h}_{2}}
 \times \cdots \times [A_{\ell_{i}}]_{\vc{h}_{i-1}, \vc{h}},\\
  &&  \quad \quad \mbox{where } i \ge 1, \ell_{m} = 0, \pm 1, m = 1,2,\ldots, i, \vc{h}_{n} \in \sr{H}, n = 1,2,\ldots, i-1 \}
\end{eqnarray*}
 is shown to be one.
}
\end{remark}

  It is not very hard to find positive vectors $\vc{x}$ and $\vc{y}$ satisfying conditions \eqn{left A-vector} and \eqn{positive recurrent}. 
We first derive the $\alpha$ and the right invariant vector $\vc{y}$ in \eqn{left A-vector}
\begin{lemma}
\label{lem:right invariant y}
{\rm
 Let ${y}_{\vc{h}} =
\rho^{-\vc{h}\vc{1}}$ for $\vc{h} \in {\sr H}$.
Then,  $\alpha = \rho^{-k}$ and $\vc{y} = ({y}_{\vc{h}};\vc{h} \in {\sr H})$ satisfy $A_{*}(\alpha) \vc{y} = \vc{y}$.  
}
\end{lemma}

This lemma is proved in \app{Right}. We next consider left invariant vector $\vc{x}$ and \eqn{positive recurrent}.  Let $\Delta_{\vc{y}}$ be the diagonal matrix whose $\vc{h}$-th diagonal element is ${y}_{\vc{h}}$ for $\vc{h} \in {\sr H}$ and the other entires are $0$. 
Since $\vc{y}$ is the right invariant vector of $A_{*}(\rho^{-k})$,
$\Delta_{\vc{y}}^{-1} A_{*}(\rho^{-k}) \Delta_{\vc{y}}$ is a stochastic matrix.
We have the following lemma, which is proved in \app{proof of positive recurrent}.

\begin{lemma}
\label{lem:positivity}
{\rm
The stochastic matrix $\Delta_{\vc{y}}^{-1} A_{*}(\rho^{-k}) \Delta_{\vc{y}}$ is positive recurrent.
}
\end{lemma}
Using these lemmas, we will verify the conditions \eqn{left A-vector} and \eqn{positive recurrent} in \sectn{Proof of Theorem}. 

We finally consider the condition \eqn{pi0 y}. However, we need much effort to check condition \eqn{pi0 y} since it includes the unknown vector $\vc{\pi}_{0}$. For this, we will consider a convergence domain of the moment generating function for the JSQ in \sectn{Stationary}. In the next subsection, we only present results on the domain as lemmas, and prove them in Appendices \appt{exponential tightness} and \appt{stationary inequality} and in \sectn{Proof of Domain}.

\subsection{Stationary inequality for moment generating functions}
\label{sec:Stationary}

	For $\vc{\theta} \equiv (\theta_{0},\theta_{1},\theta_{2},\ldots,\theta_{k} ) \in \dd{R}^{k+1}$, where $\dd{R}$ is the set of all
 real numbers, let $\varphi(\vc{\theta})$ denote the moment generating function of the random vector $\vc{Z} = (M, \vc{Y})$
 in \eqn{stationary 1}, that is,
\begin{eqnarray}
\label{eqn:moment function}
	\varphi(\vc{\theta}) = \dd{E}(e^{\vc{\theta} \vc{Z}}),
\end{eqnarray}
 where $\vc{a} \vc{b}$ denotes the inner product of vectors $\vc{a}$ and $\vc{b}$. 
	We are interested in the convergence domain of $\varphi$ which is denoted by ${\sr D}$, i.e.,
\begin{eqnarray}
\label{eqn:domain D}
{\sr D} = \{ \vc{\theta} \in \dd{R}^{k+1}; \varphi(\vc{\theta}) < \infty \}.
\end{eqnarray}

	In what follows, we will study the domain ${\sr D}$ by using the stationary equation for the moment generating function.
	To this end, we introduce some notations.
	For $U \in {\sr K}$, let $\gamma_{+U}$ and $\gamma_{0U}$ be the moment generating functions of the random vectors $\vc{X}^{(+U)}$ and $\vc{X}^{(0U)}$, respectively, that is,
\begin{eqnarray*}
	\gamma_{+U}(\vc{\theta}) = \dd{E}(e^{\vc{\theta} \vc{X}^{(+U)}}), \qquad
	\gamma_{0U}(\vc{\theta}) = \dd{E}(e^{\vc{\theta} \vc{X}^{(0U)}}), \qquad \vc{\theta} \in \dd{R}^{k+1}.
\end{eqnarray*}
	From equations \eqn{Dist of vc{X}(U)1}--\eqn{Dist of vc{X}(0,U)2}, we have
\begin{eqnarray}
\label{eqn:moment X}
&&\gamma_{+U}(\vc{\theta}) = \left \{
\begin{array}{ll}
\displaystyle{
 \lambda e^{\theta_{0} - \sum_{j \in J \setminus U} \theta_{j}}}
+ \sum_{j \in J \setminus U} \mu_{j} e^{-\theta_{j}} 
+ \mu_{i} e^{-\theta_{0} + \sum_{j \in J \setminus U} \theta_{j}},
& \mbox{$|U| = 1$, $i \in U$},\\
\displaystyle{
\sum_{i \in U} \frac{1}{|U|} \lambda e^{\theta_{i}}}
+ \sum_{i \in J \setminus U} \mu_{i} e^{-\theta_{i}} 
+ \sum_{i \in U} \mu_{i} e^{-\theta_{0} + \sum_{j \in J \setminus \{i\}} \theta_{j}}, 
& \mbox{$|U| \ge 2$},
\end{array}
\right.\\
\label{eqn:moment X0}
&&\gamma_{0U}(\vc{\theta}) = \left \{
\begin{array}{ll}
\displaystyle{
\lambda e^{\theta_{0} - \sum_{j \in J \setminus U} \theta_{j}}} 
+ \sum_{j \in J \setminus U} \mu_{j} e^{-\theta_{j}} 
+ \mu_{i},
& \mbox{$|U| = 1$, $i \in U$},\\
\displaystyle{
\sum_{i \in U} \frac{1}{|U|} \lambda e^{\theta_{i}}} 
+ \sum_{i \in J \setminus U} \mu_{i} e^{-\theta_{i}} 
+ \sum_{i \in U} \mu_{i},
& \mbox{$|U| \ge 2$}.
\end{array}
\right.
\end{eqnarray}
	We further define two moment generating functions $\varphi_{+U}$ and $\varphi_{0U}$ as follows:
\begin{eqnarray}
\label{eqn:MGF varphiU}
	\varphi_{+U}(\vc{\theta}) &=& \dd{E}(e^{\vc{\theta} \vc{Z}}1(\vc{Z} \in {\sr S}_{+U})),\\
\label{eqn:MGF varphiU0}
  \varphi_{0U}(\vc{\theta}) &=& \dd{E}(e^{\vc{\theta} \vc{Z}}1(\vc{Z} \in {\sr S}_{0U})),
\end{eqnarray}
 for $\vc{\theta} \in \dd{R}^{k+1}$.
\begin{remark}
{\rm 
\label{rem:Remark for varphiU and varphiU0}
	For each $U \in {\sr K}$ and $\vc{\theta} \in \mathbb{R}^{k+1}$, $\varphi_{+U}(\vc{\theta})$ does not depend on the parameter $\theta_{i}$ for $i \in U$ since the expectation in \eqn{MGF varphiU} is taken over the event $\{\vc{Z} \in {\sr S}_{+U}\}$, i.e., 
 $Y_{i} = 0$ for all $i \in U$.
	Similarly, $\varphi_{0U}(\vc{\theta})$ does not depend on the parameters $\theta_{0}$ and $\theta_{i}$ for $i \in U$.
}
\end{remark}

From \eqn{moment X} and \eqn{moment X0}, for all $U \in {\sr K}$, it is easy to see that $\gamma_{+U}(\vc{\theta})$ and $\gamma_{0U}(\vc{\theta})$ are finite for all $\vc{\theta} \in \mathbb{R}^{k+1}$.
	Thus, from \eqn{stationary 1}, we have a stationary equation of moment generating function as long as $\varphi(\vc{\theta})$ is finite.
\begin{eqnarray}
\label{eqn:stationary 2}
\varphi(\vc{\theta}) = \sum_{U \in {\sr K}} \left( 
\gamma_{+U}(\vc{\theta}) \varphi_{+U}(\vc{\theta}) + 
\gamma_{0U}(\vc{\theta}) \varphi_{0U}(\vc{\theta})
\right ) .
\end{eqnarray}
Furthermore, partitioning $\varphi(\vc{\theta})$ concerning each boundary, we have the following decomposition:
\begin{eqnarray}
\label{eqn:moment decomposition}
\varphi(\vc{\theta}) = \sum_{U \in {\sr K}} (\varphi_{+U} (\vc{\theta}) + \varphi_{0U} (\vc{\theta})).
\end{eqnarray}
From \eqn{stationary 2} and \eqn{moment decomposition}, we get
\begin{eqnarray}
\label{eqn:stationary 3}
\sum_{U \in {\sr K}} \left( (1-\gamma_{+U}(\vc{\theta})) \varphi_{+U}(\vc{\theta}) + (1-\gamma_{0U}(\vc{\theta})) \varphi_{0U}(\vc{\theta}) \right) = 0.
\end{eqnarray}
The stationary equation \eqn{stationary 3} holds at least $\vc{\theta} \le \vc{0}$.

We consider the convergence domain of the moment generating function $\varphi(\vc{\theta})$ for $\vc{\theta} > \vc{0}$.
We first prove that the distribution of $\vc{Z}$ has a light tail.
For this, we will use the idea of Foley and Mcdonald \cite{FM 2001}, in which they obtain a similar result for $k=2$. 
\begin{lemma}
\label{lem:exponential tightness}
{\rm
Under the stability condition \eqn{Stability condition}, there exists an $\epsilon > 0$ such that $\mathbb{E}(e^{\epsilon \sum_{i=1}^{k} L_{i}}) < \infty$, where $L_{i} = M + Y_{i}$ for $i \in J$, and therefore
$\varphi(0,\epsilon \vc{1}) < \infty$.
}
\end{lemma}

The proof of this lemma is deferred to \app{exponential tightness}.
From this lemma, we find a confirmed region for the convergence domain of the moment generating function $\varphi$ in the positive orthant. 
Starting  with this region, we will expand the confirmed region of the convergence domain ${\sr D}$.
To this end, we consider the stationary equation \eqn{stationary 3}.
However, we only know that the stationary equation \eqn{stationary 3} holds with $\vc{\theta} \le \vc{0}$.
Thus, we can not use the stationary equation \eqn{stationary 3} directly.
Instead of doing so, we derive inequalities on the stationary distribution.
For this, we recall that ${\sr K}$ is the set of all subsets of $J$ except for empty set.

\begin{lemma}
\label{lem:stationary inequality}
{\rm
For each $\vc{\theta} \in \mathbb{R}^{k+1}$, we have the following results.\\
(a) If $\varphi_{0U}(\vc{\theta}) < \infty$ and $1-\gamma_{+U}(\vc{\theta}) > 0$ for all $U \in {\sr K}$, then we have 
\begin{eqnarray}
\label{eqn:stationary inequality}
\sum_{U \in {\sr K}}(1 - \gamma_{+U}(\vc{\theta}))\varphi_{+U}(\vc{\theta}) 
 \leq \sum_{U \in {\sr K}} (\gamma_{0U}(\vc{\theta}) - 1)\varphi_{0U}(\vc{\theta}) <\infty. 
\end{eqnarray}
(b) Let ${\sr A}$ be a subset ${\sr K}$.
If
\begin{eqnarray}
\label{eqn:condition a1}
&&\gamma_{+U}(\vc{\theta}) < 1,  \gamma_{0U}(\vc{\theta}) <1, \,\,\, \qquad U \in {\sr K} \setminus {\sr A}, \\
\label{eqn:condition a2}
&&\varphi_{+U'}(\vc{\theta}) < \infty, \varphi_{0U'}(\vc{\theta})<\infty, \quad  U' \in {\sr A},
\end{eqnarray}
then we have 
\begin{eqnarray}
\label{eqn:stationary inequality2}
&&\sum_{U \in {\sr K} \setminus {\sr A}} \left( (1 - \gamma_{+U}(\vc{\theta}))\varphi_{+U}(\vc{\theta}) + (1 - \gamma_{0U}(\vc{\theta}))\varphi_{0U}(\vc{\theta}) \right) \nonumber \\
&& \qquad \leq \sum_{U' \in {\sr A}} \left( (\gamma_{+U'}(\vc{\theta}) - 1)\varphi_{+U'}(\vc{\theta}) + (\gamma_{0U'}(\vc{\theta}) - 1)\varphi_{0U'}(\vc{\theta}) \right) < \infty.
\end{eqnarray}
(c) If all the conditions in either (a) or (b) hold, then $\varphi(\vc{\theta}) < \infty$.
}
\end{lemma}
\begin{remark}
{\rm
If $\varphi_{+U}(\vc{\theta})$ and $\varphi_{0U}(\vc{\theta})$ are finite for all $U \in {\sr K}$, 
then the stationary equation \eqn{stationary 2} is satisfied, 
and we have \eqn{stationary inequality} and \eqn{stationary inequality2} with equalities.
The important claim of this lemma is that \eqn{stationary inequality} and \eqn{stationary inequality2} hold even if the finiteness of the left hand sides are unknown.
}
\end{remark}

This lemma is an adaptation of Lemma 6.4 in \cite{Miyazawa 2010}, but we prove it in \app{stationary inequality} for this paper to be selfcontained.
Using these  inequalities, we obtain a part of domain ${\sr D}$ to be sufficient for our purpose. We prove the following lemma in \sectn{Proof of Domain}.

\begin{lemma}
\label{lem:domain D}
{\rm
 For $(\eta_{0},\eta_{1}) \in \mathbb{R}^{2}$, if
\begin{eqnarray}
\label{eqn:condition the domain D}
\eta_{0} < \log \rho^{-k}, \qquad \eta_{1} < \frac{1}{k-1}\log \rho^{-k} = \log \rho^{-1} + \frac{1}{k-1}\log \rho^{-1},
\end{eqnarray}
 then we have $ \varphi(\eta_{0},\eta_{1}\vc{1}) < \infty$.
}
\end{lemma}

This fact will be used to verify the condition \eqn{pi0 y}.
We are now ready to prove \thr{JSQ main result} using 
\pro{sufficient conditions}.
\subsection{Verifying the sufficient conditions in \pro{sufficient conditions}}
\label{sec:Proof of Theorem}
From \lem{right invariant y}, we already obtain $\alpha = \rho^{-k}$ and $\vc{y} = (y_{\vc{h}})$, where ${y}_{\vc{h}} = \rho^{-\vc{h}\vc{1}}$ for $\vc{h} \in {\sr H}$. 
So, we need to check the conditions \eqn{left A-vector}, \eqn{positive recurrent} and \eqn{pi0 y}.
From \lem{positivity}, the stationary distribution of $\Delta_{\vc{y}}^{-1} A_{*}(\rho^{-k}) \Delta_{\vc{y}}$  exists, 
and denote it by $\vc{\nu}$, that is,
\begin{eqnarray*}
\vc{\nu}\Delta_{\vc{y}}^{-1} A_{*}(\rho^{-k}) \Delta_{\vc{y}}  = \vc{\nu}.
\end{eqnarray*}
Thus, the condition \eqn{left A-vector} holds with $\vc{x} =  \vc{\nu}\Delta_{\vc{y}}^{-1}$.
In addition, we have
\begin{eqnarray*}
\vc{x} \vc{y} = \vc{\nu}\Delta_{\vc{y}}^{-1} \vc{y}  = 1.
\end{eqnarray*}
Hence,  the condition \eqn{positive recurrent} is satisfied.

We finally verify the condition \eqn{pi0 y}.
We have the following equation.
\begin{eqnarray}
\label{eqn:level 0}
\vc{\pi}_{0} \vc{y} = \sum_{\vc{h} \in {\sr H}} [\vc{\pi}_{0}]_{\vc{h}} \rho^{-\vc{h}\vc{1}} =\sum_{\vc{h} \in {\sr H}} [\vc{\pi}_{0}]_{\vc{h}} e^{\log \rho^{-1}\vc{h}\vc{1}}.
\end{eqnarray}
Since $\varphi$ is the moment generating function of stationary distribution, for $(\eta_{0},\eta_{1}) \in \dd{R}^{2}$, 
\begin{eqnarray}
\label{eqn:varphi pi}
{\varphi}(\eta_{0},\eta_{1}\vc{1}) =  \sum_{n = 0}^{\infty}\sum_{\vc{h} \in {\sr H}} [\vc{\pi}_{n}]_{\vc{h}} e^{\eta_{0}n + \eta_{1} \vc{h}\vc{1}}.
\end{eqnarray}
From \eqn{level 0} and \eqn{varphi pi}, we have
\begin{eqnarray}
\label{eqn:varphi}
\vc{\pi}_{0} \vc{y} \le  {\varphi}(0,  \log \rho^{-1}\vc{1}).
\end{eqnarray}
By \lem{domain D}, we have ${\varphi}(0,  \log \rho^{-1}\vc{1}) < \infty$, and the left hand side of \eqn{varphi} is finite.
Thus, the conditions in \pro{sufficient conditions} are satisfied.
This completes the proof of \thr{JSQ main result}.

\section{The proof of \lem{domain D}}
\setnewcounter
\label{sec:Proof of Domain}
The main object of this section is to prove \lem{domain D}. 
Let 
\begin{eqnarray*}
{\sr D}^{(2)} = \{(\eta_{0},\eta_{1}) \in \mathbb{R}^{2};(\eta_{0},\eta_{1}\vc{1}) \in {\sr D}\}.
\end{eqnarray*}
That is, ${\sr D}^{(2)}$ is the projection ${\sr D}$ to the two dimensional hyper plane.
Then, for $(\eta_{0},\eta_{1}) \in \mathbb{R}^{2}$, $\varphi(\eta_{0},\eta_{1}\vc{1}) < \infty$ is equivalent to $(\eta_{0},\eta_{1}) \in {\sr D}^{(2)}$.
Thus, we will show that $(\eta_{0},\eta_{1}) \in {\sr D}^{(2)}$ under the condition \eqn{condition the domain D}.
We iteratively find a sequence of vectors $\vc{\zeta}_{1}, \vc{\zeta}_{2}, \ldots, \vc{\zeta}_{m}$ for some $m \ge 1$ such that $\vc{\zeta}_{\ell} \in {\sr D}^{(2)}$ for $1 \le \ell \le m$ and $(\eta_{0},\eta_{1}) \le  \vc{\zeta}_{\ell}$.
Our approach is similar to \cite{KM 2010, Miyazawa 2010}. 
One may think that we can still use the two dimensional iteration.
However, things are not so simple because of the $k+1$-dimensional nature of the stationary distribution.
In particular, we must consider the $k+1$ dimensional moment generating function when we expand the confirmed region in the direction of $\eta_{1}$-axis.
To overcome this issue, we prepare some technical lemmas.

\subsection{The first step for expanding the confirmed region}
\label{sec:Stationary inequality}
 We consider to iteratively expand the confirmed region for $\varphi(\vc{\theta}) < \infty$ by using \lem{stationary inequality}. For this, it is important to suitably choose $\sr{A}$ for \lem{stationary inequality}. For each $V \in {\sr K}$, let
\begin{eqnarray*}
  \sr{G}(V) = \{U \in {\sr K}; V \cap U \neq \emptyset\},
\end{eqnarray*}
and denote $\sr{K} \setminus \sr{G}(V)$ by $\sr{G}^{c}(V)$. It follows from \lem{stationary inequality} with ${\sr A} = {\sr G}(V)$ that
\begin{eqnarray}
\label{eqn:condition (a1)'}
&&\gamma_{+U}(\vc{\theta}) < 1, \gamma_{0U}(\vc{\theta}) < 1, \,\, \qquad U \in {\sr G}^{c}(V), \\
\label{eqn:condition (a2)'}
&&\varphi_{+U}(\vc{\theta}) < \infty, \varphi_{0U}(\vc{\theta}) < \infty, \quad U \in {\sr G}(V),
\end{eqnarray}
imply that $\varphi(\vc{\theta}) < \infty$.

We next consider to minimize the number of the conditions in \eqn{condition (a2)'} to be verified by restricting to the case that $|V| = 1$.
To this end, we introduce some notations. For each $i=0,1$, let $\{\zeta_{i,j}; i =0,1, j \in J\}$ be a $2 \times k$ matrix. For  $V \in {\sr K}$,
 let $\vc{\zeta}(V)$ be the $k+1$-dimensional vector whose components are give by
\begin{eqnarray}
\label{eqn:theta V}
[\vc{\zeta}(V)]_{i} = \left \{
\begin{array}{ll}
\zeta_{0,|V|}, & i=0,\\
\zeta_{1,|V|}, & i \in V,\\
0,          &  i \in J \setminus V.
\end{array}
\right.
\end{eqnarray}
 	Then, we have the following lemma.
\begin{lemma}
\label{lem:varphi V}
{\rm 
 For each $i=0,1$, assume that $\zeta_{i,j}$ is nonincreasing in $j \in J$ and
\begin{itemize}
\item[(C1)]$\gamma_{+U}(\vc{\zeta}(V)) < 1$ and $\gamma_{0U}(\vc{\zeta}(V)) < 1$ for each $V \in \sr{K} \setminus \{J\}$, that is, $V \in \sr{K}$ satisfying $|V| \le k-1$, and $U \in \sr{G}^{c}(V)$.
\end{itemize}
  Then, the condition:
\begin{itemize}
\item [(C2)] $\varphi_{+U}(\vc{\zeta}(V)) < \infty$ for each $V \in {\sr K}\setminus \{J\}$ satisfying $|V| = 1$ and each $U \in \sr{G}(V)$
\end{itemize}
 implies that
\begin{eqnarray}
\label{eqn:zeta D (2)}
	(\zeta_{0,1},\zeta_{1,k-1}) \in {\sr D}^{(2)}.
\end{eqnarray}
}
\end{lemma}

\begin{remark}
\label{rem:Condition D1}
{\rm
	For $V \in {\sr K} \setminus \{J\}$ such that $|V| = 1$, $U \in {\sr G}(V)$ if and only if $U =V$, and therefore, for this $U$, $\varphi_{+U}(\vc{\zeta}(V))$ does not depend on the elements of $\vc{\zeta}(V)$ except for first entry, i.e., $\zeta_{0,1}$.
}
\end{remark}

\begin{proof}
For $(\zeta_{0,k-1},\zeta_{1,k-1}) \in \mathbb{R}^{2}$ , from \eqn{moment decomposition}, we have
\begin{eqnarray}
\label{eqn:decomposition 2}
	{\varphi}(\zeta_{0,k-1}, \zeta_{1,k-1}\vc{1})
   = \sum_{U \in {\sr K}}
        \left(\varphi_{0U}(\zeta_{0,k-1}, \zeta_{1,k-1}\vc{1}) + \varphi_{+U}(\zeta_{0,k-1}, \zeta_{1,k-1}\vc{1})\right).
\end{eqnarray}
  Since (C1) holds for $V \in \sr{K} \setminus \{J\}$, \eqn{zeta D (2)} is obtained if we verify
\begin{eqnarray}
\label{eqn:zeta V 2}
  \varphi_{+U}(\vc{\zeta}(V)) < \infty, \qquad V \in \sr{K} \setminus \{J\}, \; U \in \sr{G}(V),
\end{eqnarray}
   by \lem{stationary inequality}.
  
Assume the condition (C2). We inductively verify \eqn{zeta V 2} on the value of $|V|$, where $1 \le |V| \le k-1$. \eqn{zeta V 2} holds for $|V| = 1$ by (C2).
	For a fixed $\ell$, where $1 < \ell < k-1$, we assume that, for $V \in {\sr K} \mbox{ satisfying } 1 \le |V| \le {\ell}$,
\begin{eqnarray}
\label{eqn:Induction on |V|}
 \varphi_{+U}(\vc{\zeta}(V)) < \infty, \quad \varphi_{0U}(\vc{\zeta}(V)) < \infty, \qquad U \in \sr{G}(V).
\end{eqnarray}
	If we can show that \eqn{Induction on |V|} for $V \in {\sr K} \mbox{ such that } |V| = \ell + 1$, then the induction is completed, and therefore we have \eqn{zeta V 2}.

Arbitrarily choose $U' \in \sr{G}(V')$ for $V' \in \sr{K}$ satisfying $|V'| = \ell + 1$. We recall that the expectations in $\varphi_{+U'}$ and $\varphi_{0U'}$ are taken over the events
 $\{Y_{i}=0 ; i \in U'\}$ and $\{M=0\} \cap \{Y_{i}=0 ; i \in U'\}$, respectively. Hence, we have
\begin{eqnarray}
\label{eqn:Equation when V'}
	\varphi_{+U'}(\vc{\zeta}(V')) \le \varphi_{+U'}(\vc{\zeta}(V' \setminus U')), \quad
	\varphi_{0U'}(\vc{\zeta}(V')) \le \varphi_{0U'}(\vc{\zeta}(V' \setminus U')), 
\end{eqnarray}
 since $\zeta_{0,|V'|} \le \zeta_{0,|V' \setminus U'|}$ and $\zeta_{1,|V'|} \le \zeta_{1,|V' \setminus U'|}$ by the nonincreasing assumption (see also \eqn{theta V}).
	We note that $|V' \setminus U'| \le {\ell}$ for $U' \in \sr{G}(V')$ since $U' \cap V' \not= \phi$.
Thus, from the induction assumption \eqn{Induction on |V|} for $|V| \le \ell$, \eqn{Induction on |V|} is satisfied for $|V| = \ell+1$.
	This completes the proof of the lemma.
\end{proof}

From this lemma, we have the following fact, which will be used to expand the confirmed region.
\begin{lemma}
{\rm
\label{lem:direction eta1}
If $(\eta_{0},0) \in {\sr D}^{(2)}$ for $\eta_{0} > 0$, then $(\eta_{0},\frac{1}{k-1}\eta_{0}) \in {\sr D}^{(2)}$.
}
\end{lemma}
\begin{proof}
We will use \lem{varphi V}.
For this, let $\zeta_{0,j} = \eta_{0}$ and $\zeta_{1,j} = \frac{1}{j} \eta_{0}$. Then, for $i=0,1$, $\zeta_{i,j}$ is nonincreasing in $j$. Moreover, from \rem{Condition D1},  we have (C2) from our assumption.
If we can verify (C1), then for $\zeta_{0,j} = \eta_{0}$ and $\zeta_{1,j} = \frac{1}{j} \eta_{0}$, all conditions of \lem{varphi V} are satisfied, and therefore, $(\eta_{0},\frac{1}{k-1}\eta_{0}) \in {\sr D}^{(2)}$. 

In what follows, we check the condition (C1).
For this, we consider the case $U \in {\sr G}^{c}(V)$. We recall that $V \cap U = \emptyset$. Thus, for $j \in J$, $\zeta_{0,j} = \eta_{0}$ and $\zeta_{1,j} = \frac{1}{j} \eta_{0}$,
\begin{eqnarray*}
\begin{array}{llll}
\displaystyle{[\vc{\zeta(V)}]_{i}}  &=& \displaystyle{\left([\vc{\zeta}(V)]_{i} 1(i \in V) + [\vc{\zeta}(V)]_{i} 1(i \in J \setminus V) \right) = 0}, & \forall i \in U,\\
\displaystyle{\sum_{i \in J \setminus U} [\vc{\zeta(V)}]_{i}}  &=& \displaystyle{\sum_{i \in J \setminus U} \left([\vc{\zeta}(V)]_{i} 1(i \in V) + [\vc{\zeta}(V)]_{i} 1(i \in J \setminus V) \right) = {\eta_{0}}}, &|U|= 1,\\
\displaystyle{\sum_{j \in J \setminus\{i\}} [\vc{\zeta}(V)]_{j}} &=& \displaystyle{\sum_{j \in J \setminus \{i\}} \left([\vc{\zeta}(V)]_{j} 1(j \in V) + [\vc{\zeta}(V)]_{j} 1(j \in J \setminus V) \right) = {\eta_{0}}}, & \forall i \in U,\\
 \displaystyle{\sum_{i \in J \setminus U} \mu_{i} e^{-[\vc{\zeta}(V)]_{i}}} &=&  \displaystyle{\sum_{i \in J \setminus U} \mu_{i}\left(e^{-[\vc{\zeta}(V)]_{i}} 1(i \in V) + e^{-[\vc{\zeta}(V)]_{i}} 1(i \notin V) \right)}\\
   &=& \displaystyle{\sum_{i \in V} \mu_{i} e^{- \frac{1}{|V|}\eta_{0}} + \sum_{i \in J \setminus(U \cup V)} \mu_{i}.}
\end{array}
\end{eqnarray*}
For any $i$ satisfying $U = \{i\}$, that is $|U| = 1$, substituting these equations into \eqn{moment X} and \eqn{moment X0}, we have 

\begin{eqnarray*}
\gamma_{+U}(\vc{\zeta}(V)) &=& 
\lambda e^{\zeta_{0,|V|} - \sum_{j \in J \setminus U} [\vc{\zeta}(V)]_{j}}
+ \sum_{j \in J \setminus U} \mu_{j} e^{-[\zeta(|V|)]_{i}} 
+\mu_{i} e^{-\zeta_{0,|V|} + \sum_{j \in J \setminus \{i\}} [\vc{\zeta}(V)]_{j}} \\
&=& \lambda + \sum_{j \in V} \mu_{j} e^{-\frac{1}{|V|} \eta_{0}} + \sum_{j \in J \setminus (U \cup V)} \mu_{j} + \mu_{i} \\
&<& \lambda + \sum_{j \in V} \mu_{j} + \sum_{j \in J \setminus (U \cup V)} \mu_{j} + \mu_{i} \\
&=& \lambda + \sum_{j=1}^{k} \mu_{j} \\
&=& 1,\\
\gamma_{0U}(\vc{\zeta}(V)) &=& \lambda e^{\zeta_{0,|V|} - \sum_{j \in J \setminus U} [\vc{\zeta}(V)]_{j}}
+ \sum_{j \in J \setminus U} \mu_{j} e^{-[\zeta(|V|)]_{j}} 
+\mu_{i}\\
&=& \lambda + \sum_{j \in V} \mu_{j} e^{-\frac{1}{|V|} \eta_{0}} + \sum_{j \in J \setminus (U \cup V)} \mu_{j} + \mu_{i}\\
&<&1,
\end{eqnarray*}
where we used the assumption \eqn{Uniform}.
On the other hand, for $|U| \ge 2$, we have
\begin{eqnarray*}
\gamma_{+U}(\vc{\zeta}(V)) &=&
\sum_{i \in U} \frac{1}{|U|} \lambda e^{[\vc{\zeta}(V)]_{i}} 
+ \sum_{i \in J \setminus U} \mu_{i} e^{-[\vc{\zeta}(V)]_{i}} 
+\sum_{i \in U} \mu_{i} e^{-\zeta_{0,|V|} + \sum_{j \in J \setminus\{i\}} [\vc{\zeta}(V)]_{j}} \\
&=& \lambda + \sum_{i \in V} \mu_{i} e^{-\frac{1}{|V|} \eta_{0}} + \sum_{i \in J \setminus (U \cup V)} \mu_{i} + \sum_{i \in U} \mu_{i} \\
&<& 1.\\
\gamma_{0U}(\vc{\zeta}(V)) &=& \sum_{i \in U} \frac{1}{|U|} \lambda e^{[\vc{\zeta}(V)]_{i}} 
+ \sum_{i \in J \setminus U} \mu_{i} e^{-[\vc{\zeta}(V)]_{i}} 
+\sum_{i \in U} \mu_{i}\\
&=& \lambda + \sum_{i \in V} \mu_{i} e^{-\frac{1}{|V|} \eta_{0}} + \sum_{i \in J \setminus (U \cup V)} \mu_{i} + \sum_{i \in U} \mu_{i}\\
&<& 1.
\end{eqnarray*}
Hence, for each fixed $V$, we have (C1). This completes the proof.
\end{proof}

\subsection{Iterations for expansion}
\label{sec:Third step}

 We next consider two dimensional marginals of the moment generating functions \eqn{moment X} and \eqn{moment X0}. For each $U \in {\sr K}$ and $(\eta_{1},\eta_{2}) \in \mathbb{R}^{2}$, we define the moment generating functions $\gamma_{+U}^{(2)}$ and $\gamma_{0U}^{(2)}$ as 
\begin{eqnarray*}
\begin{array}{llll}
\gamma_{+U}^{(2)}(\eta_{0},\eta_{1}) = \gamma_{+U}(\eta_{0},\eta_{1}\vc{1}), \quad \gamma_{0U}^{(2)}(\eta_{0},\eta_{1}) = \gamma_{0U}(\eta_{0},\eta_{1}\vc{1}).

\end{array}
\end{eqnarray*}
For $U \in {\sr K}$, let
\begin{eqnarray*}
&&{\Gamma}_{+U}^{(2)} = \{\vc{\eta} = (\eta_{0},\eta_{1})\in \dd{R}^{2}; \gamma_{+U}^{(2)}(\vc{\eta}) \le 1 \}, \\
&&\partial {\Gamma}_{+U}^{(2)} = \{\vc{\eta} = (\eta_{0},\eta_{1}) \in \dd{R}^{2}; \gamma_{+U}^{(2)}(\vc{\eta}) = 1 \}.
\end{eqnarray*}
Then, we have the following facts.
\begin{proposition}
\label{pro:moment properties}
{\rm
Under the assumption of \thr{JSQ main result}, we have the following properties.
\begin{mylist}{0}
\item[(i)] ${\Gamma}_{+U}^{(2)}$ is a convex set in $\mathbb{R}^{2}$ for $U \in {\sr K}$.
\item[(ii)] $\{\eta_{i} \in \mathbb{R} ; (\eta_{0},\eta_{1}) \in \partial \Gamma_{+U}^{(2)}\}$ is bounded for each fixed $\eta_{1-i}$ and $i=0,1$ and $U \in {\sr K}$ such that $|U| = 1$.
\item[(iii)] $(0,\eta_{1}) \in \partial \Gamma_{+U}^{(2)}$ for some  $\eta_{1} > 0$. and $U \in {\sr K}$ satisfying $|U| = 1$.
\item[(iv)]For $U \in {\sr K}$ such that $|U| \ge 2$, $\{\eta_{1} \in \mathbb{R} ; (\eta_{0},\eta_{1}) \in \partial \Gamma_{+U}^{(2)}, \exists \eta_{0} \in \mathbb{R} \}$ is bounded from above.
\item[(v)] $\partial {\Gamma}_{+U}^{(2)}$ intersects at $(0,0)$ and $(\log \rho^{-k},\log \rho^{-1})$ for each $U \in {\sr K}$. 
\end{mylist}
}
\end{proposition}
We obviously see (i)--(iv). For example, (i) is obtained because $\gamma_{+U}^{(2)}$ is a convex function. Furthermore, (v) is obtained by letting $\theta_{0} = \log \rho^{-k}$ and $\theta_{1} = \theta_{2} \dots = \theta_{k} = \log \rho^{-1}$ in \eqn{moment X}. So far, we omit a detailed proof of this proposition.
In \fig{The way to expand the finite domain}, for $k=2$, we depict the convex curve $\partial {\Gamma}_{+U}^{(2)}$ and the region where the condition \eqn{condition the domain D} holds.
\begin{figure}[htbp]
	\centering
	\includegraphics[height=0.25\textheight]{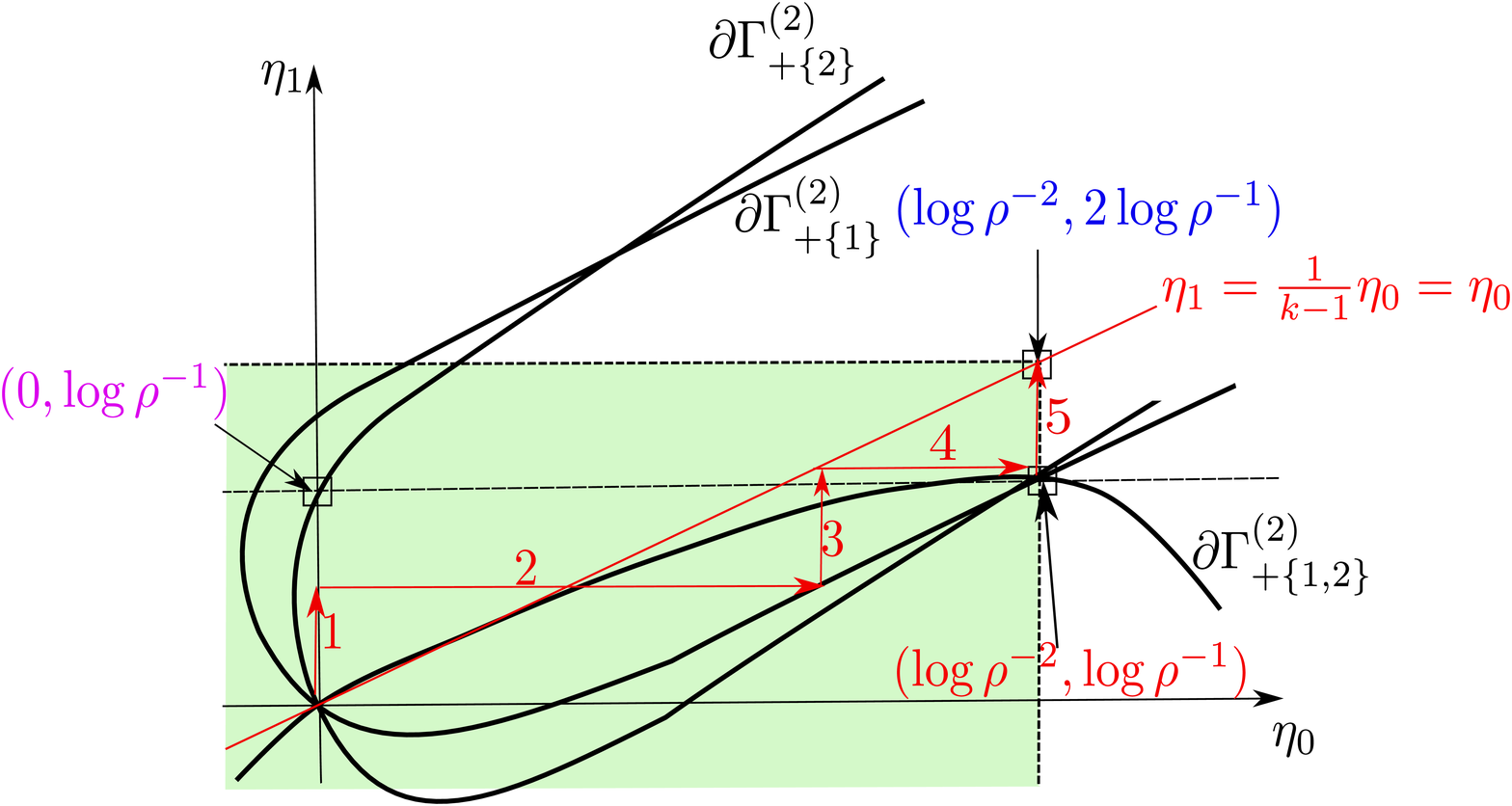}
	\caption{The domain and iteration to expand the finite domain}
	\label{fig:The way to expand the finite domain}
\end{figure}

Using \lem{direction eta1},  we iteratively find nondecreasing point such that the moment generating function $\varphi(\eta_{0},\eta_{1}\vc{1})$ is finite.
We illustrate our iteration in \fig{The way to expand the finite domain}.
For this, we recall that there exists $\ol{\eta}_{1}^{(0)} > 0$ such that $\varphi(0,\ol{\eta}_{1}^{(0)}\vc{1})$ is finite by \lem{exponential tightness}.
Let $\ol{\vc{\eta}}^{(0)} = (0,\ol{\eta}_{1}^{(0)})$ and for $\ell = 1,2,\dots $, 
\begin{eqnarray}
\label{eqn:eta 0 ell}
&&\vc{\eta}^{(\ell)} = (\eta_{0}^{(\ell)}, \eta_{1}^{(\ell)}) \equiv \arg\sup_{(\eta_{0},\eta_{1})} \{\eta_{0};\gamma_{+U}^{(2)}(\vc{\eta}) < 1, \forall U \in {\sr K}, \eta_{1} \le \ol{\eta}_{1}^{(\ell -1)} \},\\
\label{eqn:eta 1 ell}
&&\ol{\vc{\eta}}^{(\ell)} = (\ol{\eta}_{0}^{(\ell)},\ol{\eta}_{1}^{(\ell)}) \equiv \left(\eta_{0}^{(\ell)}, \frac{1}{k-1} \eta_{0}^{(\ell)}\right).
\end{eqnarray}
Then, from \pro{moment properties}, we have the following property.
\begin{lemma}
\label{lem:properties eta ell}
{\rm
$\ol{\vc{\eta}}^{(\ell)}$ is nondecreasing in $\ell = 1,2,\ldots$.
}
\end{lemma}
\begin{proof}
  From \eqn{eta 1 ell}, it is sufficient to show that $\eta_{0}^{(\ell)} \le \eta_{0}^{(\ell + 1)}$ for all $\ell = 1,2,\ldots$.
We first note that, from (i) and (iii) of \pro{moment properties}, 
\begin{eqnarray*}
\label{eqn:gamma epsilon}
\gamma_{+U}^{(2)}(0,\ol{\eta}_{1}^{(0)}) < 1, \qquad \forall U \in {\sr K} \mbox{ s.t } |U| = 1.
\end{eqnarray*}
From this inequality and the properties (i), (iv) and (v), for all $\ell = 1,2,\ldots$, it is easy to see that $0 < \eta_{0}^{(\ell)} \le \log \rho^{-k}$. 
In addition, by (i), (ii), (iv) and (v) of \pro{moment properties}, 
\begin{eqnarray}
\label{eqn:eta 0 ell2}
\vc{\eta}^{(\ell)} = \arg\sup\{\eta_{0};\gamma_{+U}^{(2)}(\vc{\eta}) < 1, U \in {\sr K},|U| = 1, 0 < \eta_{0} \le \log \rho^{-k}, \eta_{1} \le \eta_{1}^{(\ell-1)}\}.
\end{eqnarray}
From \eqn{eta 0 ell2}, if we can obtain $\eta_{1}^{(\ell)} < \ol{\eta}_{1}^{(\ell)}$ for all $\ell = 1,2,\ldots$, we have $\eta_{0}^{(\ell)} \le \eta_{0}^{(\ell + 1)}$, and the proof is completed. From convexity and boundedness of $\gamma_{+U}^{(2)}$ (see properties (i) and (ii) in \pro{moment properties}), we obtain 
\begin{eqnarray}
\label{eqn:eta 00}
&&\gamma_{+U}^{(2)}(\vc{\eta}^{(\ell)}) = 1, \quad  \exists U \in {\sr K}, |U| = 1,
\end{eqnarray}
for $\ell = 1,2,\ldots$. Thus, from convexity of $\gamma_{+U}^{(2)}$, \eqn{eta 0 ell2}, \eqn{eta 00} and $0 < \eta_{0}^{(\ell)} \le \log \rho^{-k}$, we obtain,
\begin{eqnarray}
\label{eqn:relation eta 0 and eta 1}
\eta_{1}^{(\ell)} \le \frac{1}{k} \eta_{0}^{(\ell)} < \frac{1}{k-1} \eta_{0}^{(\ell)} = \ol{\eta}_{1}^{(\ell)},
\end{eqnarray}
for all $\ell = 1,2,\ldots$. 
\end{proof}

From \lem{properties eta ell} and \eqn{eta 0 ell2}, we can see that $\ol{\eta}_{0}^{(\ell)}$ and $\ol{\eta}_{1}^{(\ell)}$ converge to some points. Denote them by $\ol{\eta}_{0}^{(\infty)}$ and $\ol{\eta}_{1}^{(\infty)}$, i.e.,
\begin{eqnarray*}
\ol{\eta}_{0}^{(\infty)} = \lim_{\ell \to \infty} \ol{\eta}_{0}^{(\ell)}, \quad 
\ol{\eta}_{1}^{(\infty)} = \lim_{\ell \to \infty} \ol{\eta}_{1}^{(\ell)}.
\end{eqnarray*}
Then, the finite domain of the moment generating function is obtained as follows.

\begin{lemma}
\label{lem:expand the domain D}
{\rm
If $(\eta_{0},\eta_{1}) < (\ol{\eta}_{0}^{(\infty)}, \ol{\eta}_{1}^{(\infty)})$ , then $(\eta_{0},\eta_{1}) \in {\sr D}^{(2)}$.
}
\end{lemma}
\begin{proof}
By induction for $\ell$, we will show that, for any $\ell = 1,2,\dots$,
\begin{eqnarray}
\label{eqn:induction_ell}
(\eta_{0},\eta_{1}) \in {\sr D}^{(2)}, \qquad (\eta_{0},\eta_{1}) < (\ol{\eta}_{0}^{(\ell)}, \ol{\eta}_{1}^{(\ell)}).
\end{eqnarray}
For $\ell = 1$, we first show that $(\eta_{0},\eta_{1}) \in {\sr D}^{(2)}$ for $(\eta_{0},\eta_{1}) < ({\eta}_{0}^{(1)},{\eta}_{1}^{(1)})$.
 For all $U \in {\sr K}$, $\varphi_{0U}(\eta_{0}^{(1)},\eta_{1}^{(1)}\vc{1})$ is finite since the value of $\varphi_{0U}$ does not depend on the first entry (see  \rem{Remark for varphiU and varphiU0}) and $\eta_{1}^{(1)} \le \ol{\eta}_{1}^{(0)}$.
From this and \eqn{eta 0 ell}, for sufficiently small $\epsilon_{0} ,\epsilon_{1} > 0$ and $(\eta_{0},\eta_{1}) = (\eta_{0}^{(1)} -\epsilon_{0},\eta_{1}^{(1)} - \epsilon_{1})$, we can use (a) of  \lem{stationary inequality}, and we have $\varphi(\eta_{0}^{(1)} -\epsilon_{0},(\eta_{1}^{(1)} - \epsilon_{1})\vc{1}) < \infty$. Hence, $(\eta_{0},\eta_{1}) \in {\sr D}^{(2)}$ for $(\eta_{0}, \eta_{1}) < (\eta_{0}^{(1)},\eta_{1}^{(1)})$.
We are ready to obtain \eqn{induction_ell} for $\ell = 1$.
We recall that
\begin{eqnarray}
\label{eqn:over_eta_1}
(\ol{\eta}_{0}^{(1)},\ol{\eta}_{1}^{(1)}) = \left(\eta_{0}^{(1)}, \frac{1}{k-1} \eta_{0}^{(1)}\right).
\end{eqnarray}
In addition,  the conditions in \lem{direction eta1} hold with $\eta_{0} < \eta_{0}^{(1)}$ since $\eta_{1}^{(1)} > 0$. Thus, from \lem{direction eta1} and \eqn{over_eta_1}, we have \eqn{induction_ell} for $\ell = 1$.

Suppose that \eqn{induction_ell} satisfies for $\ell = \ell' \ge 2$.
Similarly to the arguments for $\ell = 1$, we have $(\eta_{0},\eta_{1}) \in {\sr D}^{(2)}$ for $(\eta_{0},\eta_{1}) < (\eta_{0}^{(\ell' + 1)},\eta_{1}^{(\ell' + 1)})$. Thus, conditions in \lem{direction eta1} hold with $\eta_{0} < \eta_{0}^{(\ell' + 1)}$ again, and we obtain \eqn{induction_ell} for $\ell = \ell'+1$.
\end{proof}

\subsection{The last step of the proof}
\label{sec:Final step}
By \lem{expand the domain D} and $\ol{\eta}_{1}^{(\infty)} = \frac{1}{k-1} \ol{\eta}_{0}^{(\infty)}$, it is sufficient to prove $\ol{\eta}_{0}^{(\infty)} = \log \rho^{-k}$. 
From \eqn{eta 0 ell2}, we clearly have 
\begin{eqnarray}
\label{eqn:eta 0 infty less than log rho k}
\eta_{0}^{(\infty)} \le \log \rho^{-k}.
\end{eqnarray}
Suppose that $\ol{\eta}_{0}^{(\infty)} = \alpha < \log \rho^{-k}$.
Then, it is easy to see that  
\begin{eqnarray*}
\label{eqn:limit_alpha}
&&\gamma_{+U}^{(2)}(\alpha, \eta_{1}^{(\ell')}) = 1, \qquad \exists U \in {\sr K}, |U| = 1,\\
&&\gamma_{+U'}^{(2)}(\alpha, \eta_{1}^{(\ell')}) < 1, \qquad \forall U' \neq U, |U'|= 1,
\end{eqnarray*}
for some $\ell' \ge 1$.
Thus, from \eqn{relation eta 0 and eta 1}, we have 
\begin{eqnarray*}
\frac{1}{k}\alpha \ge \eta_{1}^{(\ell')}.
\end{eqnarray*}
Moreover, from \pro{moment properties} and \eqn{eta 0 ell2},  for any $U \in {\sr K}$,  there exists a small $\epsilon > 0$ such that
\begin{eqnarray*}
\gamma_{+U}^{(2)}(\alpha - \epsilon, \eta_{1}^{(\ell')}) < 1.
\end{eqnarray*}
From \eqn{eta 1 ell}, 
\begin{eqnarray*}
\ol{\eta}_{1}^{(\ell')} = \frac{1}{k-1} (\alpha - \epsilon) > \frac{1}{k} \alpha \ge \eta_{1}^{(\ell')},
\end{eqnarray*}
since $\epsilon$ is sufficiently small.
Thus, by \pro{moment properties}, there exist $\delta > 0$ and $\eta_{1} \le \eta_{1}^{(\ell')}$ such that 
\begin{eqnarray*}
\gamma_{+U}^{(2)}(\alpha + \delta, \eta_{1}) < 1, \qquad  \forall U \in {\sr K}.
\end{eqnarray*}
This is a contradiction. From this and \eqn{eta 0 infty less than log rho k}, we obtain $\ol{\eta}_{0}^{(\infty)} = \log \rho^{-k}$.

\section{Concluding remarks}
\setnewcounter
\label{sec:concluding remark}
First of all, we note that our assumption on the tie break can be relaxed. 
We have assumed that arriving customers choose one of the shortest queues with equal probabilities. This assumption makes our arguments simpler, but is not essential. Namely,
for each configuration of the shortest queues, we can replace it by any distribution.

In this paper, for the $M/M$-JSQ with $k$ parallel queues, we obtained the exact tail asymptotics of the stationary distribution for the minimum queue length given the differences of queue length between each queue and minimum queue. 
It may be interesting to find $c_{\vc{h}}$ of \eqn{exact asymptotic JSQ} to see the joint distribution of the background state in the asymptotic formula. 
For this, we need to derive the left invariant vector $\vc{x}$ in \eqn{left A-vector} because $c_{\vc{h}}$ is proportional to the $\vc{h}$-th entry of $\vc{x}$. 
For $k=2$, this left invariant vector is obtained in \cite{HMZ 2007}. 
 However, this computation is difficult for the case $k \ge 3$. We leave it as an open problem.

In our proof of \thr{JSQ main result}, we obtained the subset of the convergence domain ${\sr D}$. This subset is still useful as we have seen in the proof of \cor{JSQ rough marginal}. However,
if we completely derive the domain ${\sr D}$,  we could do much better job. For example, we may have another type of tail asymptotics, e.g., joint queue length and marginal distributions.
However, this would be a hard problem since our reflecting random walk is multidimensional. This challenging problem is left for future work.

Another challenging problem is to generalize the $M/M$-JSQ to have dedicated streams. For $k=2$, the rough asymptotics of the minimum queue length have been completely obtained for this generalized model in \cite{Miyazawa two side}. We can use the present formulations by a QBD process and a reflecting random walk. However, even for $k=2$, the answer is very complicated because $A_{*}(z)$ of the QBD process may not be positive recurrent. Thus, the random walk approach may be more suitable for such a generalization.
 
Including this generalization, various modifications of the $M/M$-JSQ can be described by a multidimensional reflecting random walk with skip free jumps. Thus, it is very interesting to solve the tail asymptotic problems for a general multidimensional reflecting random walk. Some related results can be found in \cite{KM 2010, Miyazawa 2010}. The approach in this paper may be useful to get the rough and exact asymptotics of the multidimensional reflecting random walk.

\section*{Acknowledgements}
The authors would like to thank the referees for their valuable comments and suggestions.
This research was supported in part by Japan Society for the Promotion of Science under grant No.\ 24310115.

\appendix

\section{Proof of \lem{right invariant y}}
\setnewcounter
\label{app:Right}
 Let $S(\vc{h}) = \{i \in J ; \vc{h} \in {\sr H}, h_{i} = 0\}$, that is, the set of severs having minimum queue.
For its proof, we give detailed form of $A_{i}$ for $i=0, \pm1$.
For $\vc{h}, \vc{h}' \in {\sr H}$, we recall that $[A_{i}]_{\vc{h},\vc{h}'}$ is the $(\vc{h},\vc{h}')$-th block of $A_{i}$.

\begin{mylist}{2}
\item[(a)] Entries of $A_{+1}$ :  In this case, the minimum queue length increases by $1$. This implies $|S(\vc{h})| = 1$. If $S(\vc{h}) = \{i\}$, arriving customers join the queue $i$, and the difference of queue length between queue $j$ and the minimum queue decreases by $1$ for $j \neq i$.
 Hence, we have
\begin{eqnarray*}
[A_{+1}]_{\vc{h},\vc{h} - \vc{1} + \vcn{e}_{i}} = \lambda, \qquad \vc{h} \in {\sr H}, | S(\vc{h})| = 1, i \in  S(\vc{h}).
\end{eqnarray*}
\item[(b)]  Entries of $A_{0}$ : 
The level is unchanged. When a customer arrives, it is required that $|S(\vc{h})| \neq 1$ and the arriving customer joins queue $i \in S(\vc{h})$ with probability $|S(\vc{h})|^{-1}$, which changes the background  state from $\vc{h}$ to $\vc{h}+\vcn{e}_{i}$. 
Hence, we have 
\begin{eqnarray*}
[A_{0}]_{\vc{h},\vc{h} + \vcn{e}_{i}} = | S(\vc{h})|^{-1} \lambda, \qquad \vc{h} \in {\sr H}, | S(\vc{h})| \neq 1, i \in  S(\vc{h}).
\end{eqnarray*}
When a customer finishes service, it must be at queue $j \in J \setminus  S(\vc{h})$, by which the difference of queue length between queue $j$ and the shortest queue decreases by $1$. Hence, we have
\begin{eqnarray*}
[A_{0}]_{\vc{h},\vc{h} - \vcn{e}_{j}} =  \mu_{j}, \qquad \vc{h} \in {\sr H}, j \in J \setminus  S(\vc{h}).
\end{eqnarray*}
\item[(c)]  Entries of $A_{-1}$ : This case implies minimum queue length decreases by $1$. That is, service completes at queue $i \in  S(\vc{h})$,
by which the background state changes from $\vc{h}$ to $\vc{h} + \vc{1} - \vcn{e}_{i}$. Hence, we have
\begin{eqnarray*}
[A_{-1}]_{\vc{h},\vc{h} + \vc{1} - \vcn{e}_{i}} = \mu_{i}, \qquad \vc{h} \in {\sr H}, i \in  S(\vc{h}).
\end{eqnarray*}
\end{mylist}

\begin{proof*}{Proof of \lem{right invariant y}}
For $\vc{h} \in {\sr H}$, $\vc{y} = ({y}_{\vc{h}})$ and $ S(\vc{h}) \in {\sr K}$ satisfying $| S(\vc{h})| = 1$ and $i \in S(\vc{h})$, from (a), (b) and (c), the $\vc{h}$-th element of $A_{*}(z)\vc{y}$ is given by
\begin{eqnarray*}
[A_{*}(z)\vc{y}]_\vc{h} =
z^{-1}  \mu_{i}\vc{y}_{\vc{h} + \vc{1} - \vcn{e}_{i}}  + 
\sum_{j \neq i}  \mu_{j} \vc{y}_{\vc{h} - \vcn{e}_{j}} + 
z \lambda \vc{y}_{\vc{h} - \vc{1} + \vcn{e}_{i}}.
\end{eqnarray*}
Substituting $z = \rho^{-k}$ and ${y}_{\vc{h}} =  \rho^{-\vc{h}\vc{1}}$ into this equation, we have
\begin{eqnarray*}
[A_{*}(\rho^{-k})\vc{y}]_\vc{h} = 
&=& \rho^{k}  \mu_{i} \rho^{-\vc{h}\vc{1} - (k-1)} +   \sum_{j \neq i}\mu_{j} \rho^{-\vc{h}\vc{1} + 1} + \rho^{-k} \lambda \rho^{-\vc{h}\vc{1} + (k-1)}\\
&=& \rho^{-\vc{h}\vc{1}}\left(\sum_{j=1}^{k} \mu_{j} \rho + \lambda \rho^{-1} \right)\\ 
&=& {y}_{\vc{h}},
\end{eqnarray*}
where we use our assumption \eqn{Uniform} for the last equality.

We next consider $2 \le |S(\vc{h})|  \le k$.
Then, from (b) and (c), we have, for $\vc{h} \in {\sr H}$,
\begin{eqnarray*}
[A_{*}(z)\vc{y}]_\vc{h} =
\sum_{i \in  S(\vc{h})} z^{-1}\mu_{i} \vc{y}_{\vc{h} + \vc{1} - \vcn{e}_{i}}  + 
\sum_{j \in J \setminus  S(\vc{h})} \mu_{j} \vc{y}_{\vc{h} - \vcn{e}_{j}} + 
\sum_{i \in  S(\vc{h})} \frac{1}{| S(\vc{h})|} \lambda \vc{y}_{\vc{h}+ \vcn{e}_{i}}.
\end{eqnarray*}
Hence, we have 
\begin{eqnarray*}
[A_{*}(\rho^{-k})\vc{y}]_\vc{h} &=&   \sum_{i \in S(\vc{h})} \rho^{k} \mu_{i} \rho^{-\vc{h}\vc{1} - (k-1)} +  \sum_{j \notin S(\vc{h})} \mu_{j} \rho^{-\vc{h}\vc{1} + 1} + \lambda \rho^{-\vc{h}\vc{1} - 1}\\
&=&  \rho^{-\vc{h}\vc{1}} \left( \sum_{j=1}^{k} \mu_{j} \rho + \lambda \rho^{-1} \right) \\
&=& y_{\vc{h}}.
\end{eqnarray*}
\end{proof*}

\section{Proof of \lem{positivity}}
\setnewcounter
\label{app:proof of positive recurrent}

From the Foster's theorem (see e,g., \cite{Bremaud,Sean Meyn 2009}),  for some $\epsilon > 0$, it is sufficient to find a function $f$ and finite set ${\sr F} \subset {\sr H}$ such that
\begin{eqnarray}
\label{eqn:Lyapunov 1}
&&\inf_{\vc{h} \in {\sr H}} f(\vc{h}) > -\infty,\\
\label{eqn:Lyapunov 2}
&&\sum_{\vc{h}' \in {\sr H}} p_{\vc{h},\vc{h}'} f(\vc{h}') < \infty,  \quad \forall \vc{h} \in  {\sr F},\\
\label{eqn:Lyapunov 3}
&&\sum_{\vc{h}' \in {\sr H}} p_{\vc{h},\vc{h}'} f(\vc{h}') - f(\vc{h}) \le -\epsilon, \quad \forall \vc{h} \in {\sr H} \setminus {\sr F},
\end{eqnarray}
where $p_{\vc{h},\vc{h}'}$ are the ($\vc{h},\vc{h}'$)-th entry of $\Delta_{\vc{y}}^{-1} A_{*}(\rho^{-k}) \Delta_{\vc{y}}$. 
The entries of the stochastic matrix $\Delta_{\vc{y}}^{-1} A_{*}(\rho^{-k}) \Delta_{\vc{y}}$
are given by the following  forms.
\begin{itemize}
\item For $|S(\vc{h})| = 1$, 
\begin{eqnarray*}
p_{\vc{h},\vc{h}'} = \left \{ 
\begin{array}{ll}
\lambda \rho^{-1}, & \mbox{$\vc{h}' = \vc{h} - \vc{1} + \vcn{e}_{i}$, $i \in S(\vc{h})$},\\
\mu_{i} \rho, & \mbox{$\vc{h}' = \vc{h} + \vc{1} - \vcn{e}_{i}$, $i \in S(\vc{h})$}\\
              & \mbox{or $\vc{h}' =\vc{h} - \vcn{e}_{i}$, $i \in J \setminus S(\vc{h})$,}\\
0,                  & \mbox{otherwise.}
\end{array}
\right.
\end{eqnarray*}
\item For $2 \le |S(\vc{h})| \le k$,
\begin{eqnarray*}
p_{\vc{h},\vc{h}'} = \left \{ 
\begin{array}{ll}
\frac{\lambda \rho^{-1}}{|S(\vc{h})|}, & \mbox{$\vc{h}' = \vc{h} + \vcn{e}_{i}$, $i \in S(\vc{h})$, }\\
\mu_{i} \rho, & \mbox{$\vc{h}' = \vc{h} + \vc{1} - \vcn{e}_{i}$, $i \in S(\vc{h})$}\\
 & \mbox{or $\vc{h}' = \vc{h} - \vcn{e}_{i}$, $i \in J \setminus S(\vc{h})$,}\\
0,                  & \mbox{otherwise.}
\end{array}
\right.
\end{eqnarray*}
\end{itemize}

For $\vc{h} = (h_{1},h_{2},\ldots,h_{k}) \in {\sr H}$, let 
\begin{eqnarray*}
f(\vc{h}) = \frac{1}{2} \sum_{j=1}^{k-1} \sum_{m=j+1}^{k} (h_{j} - h_{m})^{2}.
\end{eqnarray*}
For this $f$, we obviously have \eqn{Lyapunov 1}. 
We note that $i \in S(\vc{h})$ implies $h_{i} = 0$.
For $i \in S(\vc{h})$ and $\vc{h}' = \vc{h} - \vc{1} + \vc{e}_{i}$ or $\vc{h}' = \vc{h} + \vc{e}_{i}$, we have

\begin{eqnarray*}
(h'_{j} - h'_{m})^{2}  =\left\{
\begin{array}{ll}
 (h_{j} - h_{m})^{2}, & j,m \neq i,\\
(h_{j} - h_{m})^{2} - (2h_{j}-1), & j < m = i,\\
(h_{j} - h_{m})^{2} - (2h_{m} - 1), & j = i < m.
\end{array}
\right.
\end{eqnarray*}
Thus, we have 
\begin{eqnarray*}
p_{\vc{h},\vc{h'}} f(\vc{h}') = &=& \frac{\lambda \rho^{-1}}{2|S(\vc{h})|} \left(\sum_{j=1}^{k-1} \sum_{m=j+1}^{k} (h_{j} - h_{m})^{2} - \sum_{j=1}^{i-1} \left( 2h_{j} -1 \right) - \sum_{m=i+1}^{k}\left( 2h_{m} - 1 \right)\right)\\
&=& \frac{\lambda \rho^{-1}}{|S(\vc{h})|} \left(f(\vc{h}) - \sum_{j=1}^{k} h_{j} + \frac{k-1}{2}\right).
\end{eqnarray*}
Similarly, for $\vc{h'} = \vc{h} + \vc{1} - \vcn{e}_{i}$ and $i \in S(\vc{h})$,  
\begin{eqnarray*}
p_{\vc{h},\vc{h}'} f(\vc{h}') = \mu_{i} \rho \left(f(\vc{h}) + \sum_{j=1}^{k} h_{j} + \frac{k-1}{2}\right).
\end{eqnarray*}
For $\vc{h'} = \vc{h} - \vcn{e}_{i}$ and $i \in J \setminus S(\vc{h})$, we also obtain,
\begin{eqnarray*}
(h'_{j} - h'_{m})^{2} &=&  \left\{ 
\begin{array}{ll}
(h_{j} - h_{m})^{2}, & j,m \neq i,\\
(h_{j} - h_{m})^{2} + 2(h_{j} - h_{m}) + 1, & j < m =i,\\
(h_{j} - h_{m})^{2} + 2(h_{m} - h_{j}) + 1, & j=i < m,
\end{array}
\right.
\end{eqnarray*}
so we have the following inequality.
\begin{eqnarray*}
p_{\vc{h},\vc{h'}} f(\vc{h}') &=& \mu_{i} \rho \left(f(\vc{h}) + \sum_{j = 1}^{i-1} \left( h_{j} - h_{i} + \frac{1}{2} \right) + \sum_{m = i+1}^{k} \left( h_{m} - h_{i} + \frac{1}{2} \right) \right) \\
& \le& \mu_{i} \rho \left(f(\vc{h}) + \sum_{j = 1}^{k} h_{j} + \frac{k-1}{2} \right).
\end{eqnarray*}
From \eqn{Uniform}, we have
\begin{eqnarray*}
\sum_{\vc{h}' \in {\sr H}} p_{\vc{h},\vc{h}'} f(\vc{h}') &\le& \sum_{i \in S(\vc{h})}\frac{\lambda \rho^{-1}}{|S(\vc{h})|} \left(f(\vc{h}) - \sum_{j=1}^{k} h_{j} + \frac{k-1}{2}\right) \\
&& + \sum_{i \in S(\vc{h})} \mu_{i} \rho \left(f(\vc{h}) + \sum_{j=1}^{k} h_{j} + \frac{k-1}{2}\right)\\
&&+ \sum_{i \in J \setminus S(\vc{h})} \mu_{i} \rho \left(f(\vc{h}) + \sum_{j=1}^{k} h_{j} + \frac{k-1}{2}  \right)\\
&=& f(\vc{h}) + \left(\lambda - \sum_{j=1}^{k} \mu_{j} \right)\vc{h}\vc{1} + \frac{k-1}{2}.
\end{eqnarray*}
Thus, from the stability condition \eqn{Stability condition}, \eqn{Lyapunov 2} and \eqn{Lyapunov 3} hold with 
\begin{eqnarray*}
{\sr F} = \left\{\vc{h} \in \sr{H}; \vc{h}\vc{1} \le \frac{k-1 + \epsilon}{2(\sum_{i=1}^{k} \mu_{i} - \lambda)} \right\}.
\end{eqnarray*}
We complete the proof.

\section{Proof of \lem{exponential tightness}}
\setnewcounter
\label{app:exponential tightness}
 For the proof of the first statment in \lem{exponential tightness}, it suffices to  prove that there exists $\alpha$ such that 
\begin{eqnarray}
\label{eqn:alpha tightness}
\mathbb{E}(e^{\alpha \sqrt{\sum_{i=1}^{k}L_{i}^{2}}}) < \infty,
\end{eqnarray}
since 
$\mathbb{E}(e^{\frac{\alpha}{\sqrt{k}} \sum_{i=1}^{k} L_{i}}) \le 
\mathbb{E}(e^{\alpha \sqrt{\sum_{i=1}^{k}L_{i}^{2}}})$.
To prove \eqn{alpha tightness},  we consider a Lyapunov function such that
\begin{eqnarray*}
\label{eqn:Lyapunov exponential tight}
f(\vc{u}) = e^{\alpha \sqrt{\sum_{i=1}^{k} u_{i}^{2}}},
\end{eqnarray*}
for $\vc{u} = (u_{1},u_{2},\cdots,u_{k}) \in \mathbb{Z}_{+}^{k}$, and for each $\beta > 0$, define   
\begin{eqnarray*}
F_{\beta} = \left\{\vc{u} \in \mathbb{Z}_{+}^{k};\sum_{i=1}^{k} u_{i} \le \beta \right\}.
\end{eqnarray*}
Obviously, $F_{\beta}$ is a finite set. To verify \eqn{alpha tightness},  if we can show that
\begin{eqnarray}
\label{eqn:Lyapunov exponential tight2}
 \mathbb{E}(f(\vc{L}_{n+1}) | \vc{L}_{n} = \vc{u}) - f(\vc{u}) \le -c_{1}f(\vc{u}), \qquad \forall \vc{u} \in \mathbb{Z}_{+}^{k} \setminus F_{\beta},
\end{eqnarray}
for some  $c_{1} > 0$, then, for any $\vc{u} \in \mathbb{Z}_{+}^{k}$,  
\begin{eqnarray*}
\mathbb{E}(f(\vc{L}_{n+1}) | \vc{L}_{n} = \vc{u}) - (1 - c_{1})f(\vc{u}) \le s(\vc{u}),
\end{eqnarray*}
where $s$ is a function defined
\begin{eqnarray*}
s(\vc{u}) = (\mathbb{E}(f(\vc{L}_{n+1}) | \vc{L}_{n} = \vc{u}) - (1 - c_{1})f(\vc{u}))^{+} 1(\vc{u} \in F_{\beta}).
\end{eqnarray*}
Thus, by Theorem $14.3.7$ of \cite{Sean Meyn 2009}, we have
\begin{eqnarray*}
 c_{1}\sum_{\vc{u} \in \mathbb{Z}_{+}^{k}} f(\vc{u}) \mathbb{P}(\vc{L} = \vc{u}) =  c_{1}\mathbb{E}(e^{\alpha \sqrt{\sum_{i=1}^{k}L_{i}^{2}}}) \le \sum_{\vc{u} \in \mathbb{Z}_{+}^{k}} s(\vc{u}) \mathbb{P}(\vc{L} = \vc{u}) < \infty,
\end{eqnarray*}
which implies \eqn{alpha tightness}.
We will show that, for $0< \alpha  < \sqrt{k} \log \rho^{-1}$, we can find $\beta > 0$ such that \eqn{Lyapunov exponential tight2} holds.

 For $\vc{u} \in \mathbb{Z}_{+}$,  let $\vc{u}_{\ell} = (u_{\ell_{1}},u_{\ell_{2}}, \dots, u_{\ell_{k}})$ be a permutation of $\vc{u}$ satisfying
\begin{eqnarray}
\label{eqn:Lyapunov condition1}
&&u_{\ell_{1}} \le u_{\ell_{2}} \le \cdots \le u_{\ell_{k}}. 
\end{eqnarray}
We exclude the case $u_{\ell_{1}} = u_{\ell_{2}} = \cdots = u_{\ell_{k}} = 0$, that is, we assume the following condition.
\begin{eqnarray}
\label{eqn:Lyapunov condition2}
& 0 < u_{\ell_{j}}, &\exists j \in J. 
\end{eqnarray}
Then, for $j \in J$ satisfying \eqn{Lyapunov condition2}, it follow from the transition matrix of the original queue length process $\{\vc{L}_{\ell};\ell \in \mathbb{Z}_{+}\}$ that 
\begin{eqnarray*}
\mathbb{E}(f(\vc{L}_{n+1}) | \vc{L}_{n} = \vc{u}) - f(\vc{u}) &=& \lambda e^{\alpha \sqrt{(u_{\ell_{1}}+1)^{2}+ \sum_{i=2}^{k} u_{\ell_{i}}^{2}}} + \sum_{i = j}^{k} \mu_{\ell_{i}} e^{\alpha \sqrt{(u_{\ell_{i}}-1)^{2} + \sum_{m \neq i} u_{\ell_{m}}^{2}}} \\ 
&& + \sum_{i=1}^{j-1} \mu_{\ell_{i}} e^{\alpha \sqrt{\sum_{m=1}^{k} u_{\ell_{m}}^{2}}} - e^{\alpha \sqrt{\sum_{i=1}^{k} u_{\ell_{i}}^{2}}}\\
&=& \lambda e^{\alpha \sqrt{(u_{\ell_{1}}+1)^{2}+ \sum_{i=2}^{k} u_{\ell_{i}}^{2}}} + \sum_{i = j}^{k} \mu_{\ell_{i}} e^{\alpha \sqrt{(u_{\ell_{i}}-1)^{2} + \sum_{m \neq j} u_{\ell_{m}}^{2}}} \\ 
&& -\left(\lambda + \sum_{i=j}^{k}\mu_{\ell_{i}} \right)e^{\alpha \sqrt{\sum_{i=1}^{k} u_{\ell_{i}}^{2}}},
\end{eqnarray*}
where we have used \eqn{Uniform} to get the second equality.
Note that $\sqrt{1 + u} \le 1 + \frac{u}{2}$ for $u \ge -1$, so we have the following inequalities.
\begin{eqnarray*}
\sqrt{(u_{\ell_{1}}+1)^{2}+ \sum_{i=2}^{k} u_{\ell_{i}}^{2}} &=& r \sqrt{1 + \frac{1 + 2 u_{\ell_{1}}}{r^{2}}} \\
&\le& r + \frac{1 + 2 u_{\ell_{1}}}{2r}, \\
\sqrt{(u_{\ell_{i}} - 1)^{2}+ \sum_{m \neq i} u_{\ell_{m}}^{2}} &=& r\sqrt{1 + \frac{1 - 2 u_{\ell_{i}}}{ r^{2}}} \\
&\le& r + \frac{1 - 2 u_{\ell_{i}}}{2r},
\end{eqnarray*}
where $r = \sqrt{\sum_{i=1}^{k} u_{\ell_{i}}^{2}}$. Hence, we have 
\begin{eqnarray*}
\label{eqn:conditional expectation f}
\lefteqn{\mathbb{E}(f(\vc{L}_{n+1}) | \vc{L}_{n} = \vc{u})  - f(\vc{u}))}  \nonumber\\
&& \le \lambda e^{\alpha r + \frac{\alpha(1 + 2 u_{\ell_{1}})}{2r}} +  \sum_{i = j}^{k} \mu_{\ell_{i}} e^{\alpha r + \frac{\alpha(1 - 2 u_{\ell_{i}})}{2r}} -\left(\lambda + \sum_{i=j}^{k}\mu_{\ell_{i}} \right) e^{\alpha r} \nonumber\\
&& = e^{\alpha r} \left( e^{\frac{\alpha}{2r}} \left( \lambda e^{\frac{\alpha u_{\ell_{1}}}{r}} +  \sum_{i = j}^{k} \mu_{\ell_{i}} e^{-\frac{\alpha u_{\ell_{i}}}{r}} \right) - \left(\lambda + \sum_{i=j}^{k}\mu_{\ell_{i}} \right) \right) \nonumber\\
&& = f(\vc{u}) \left( e^{\frac{\alpha}{2r}} \left( \lambda e^{\frac{\alpha u_{\ell_{1}}}{r}} +  \sum_{i = j}^{k} \mu_{\ell_{\ell_{i}}} e^{-\frac{\alpha u_{\ell_{i}}}{r}} \right) - \left(\lambda + \sum_{i=j}^{k}\mu_{\ell_{i}} \right) \right).
\end{eqnarray*}
From this inequality, for a small $\alpha > 0$ and a large $\beta > 0$, it is enough to find $c_{1} > 0$ such that
\begin{eqnarray}
\label{eqn:Lyapunov sufficient}
 e^{\frac{\alpha}{2r}} \left( \lambda e^{\frac{\alpha u_{\ell_{1}}}{r}} +  \sum_{i = j}^{k} \mu_{\ell_{\ell_{i}}} e^{-\frac{\alpha u_{\ell_{i}}}{r}} \right) - \left(\lambda + \sum_{i=j}^{k}\mu_{\ell_{i}} \right) \le -c_{1}, \quad \forall \vc{u} \in \mathbb{Z}_{+} \setminus {\sr F}_{\beta}.
\end{eqnarray}
For this, for a small $\delta > 0$, we partition two cases $\frac{\alpha u_{\ell_{1}}}{r} < \delta$ and $\frac{\alpha u_{\ell_{1}}}{r} \ge \delta$.
From our assumption \eqn{Lyapunov condition1} and \eqn{Lyapunov condition2}, we have 
\begin{eqnarray}
\label{eqn:inequality u i}
&&0 \le \frac{u_{\ell_{i}}}{r} \le 1, \qquad i=1,2,\ldots,k,\\
\label{eqn:inequality u 1}
&&0 \le \frac{u_{\ell_{1}}}{r} \le  \frac{1}{\sqrt{k}},\\
\label{eqn:inequality u k}
&&\frac{1}{\sqrt{k}} \le \frac{u_{\ell_{k}}}{r} \le 1.
\end{eqnarray}
First assume that $\frac{\alpha u_{1}}{r} < \delta$. In this case, from \eqn{inequality u i} and \eqn{inequality u k}, we obtain
\begin{eqnarray*}
&&e^{\frac{\alpha}{2r}} \left( \lambda e^{\frac{\alpha u_{\ell_{1}}}{r}} +  \sum_{i = j}^{k} \mu_{\ell_{\ell_{i}}} e^{-\frac{\alpha u_{\ell_{i}}}{r}} \right) - \left(\lambda + \sum_{i=j}^{k}\mu_{\ell_{i}} \right)\\
&&\qquad \qquad \le  e^{\frac{\alpha}{2r}} \left( \lambda e^{\frac{\alpha u_{\ell_{1}}}{r}} +  \sum_{i = j}^{k-1} \mu_{\ell_{i}} + \mu_{\ell_{k}} e^{-\frac{\alpha}{\sqrt{k}}} \right) - \left(\lambda + \sum_{i=j}^{k}\mu_{\ell_{i}} \right)\\
&& \qquad \qquad = \lambda\left(e^{\frac{\alpha}{2r} + \delta} - 1 \right)  - \mu_{\ell_{k}}\left(1 - e^{-\frac{\alpha}{\sqrt{k}}} \right).
\end{eqnarray*}
Then, for small $\alpha$ satisfying $0 < \alpha < \sqrt{k} \log \rho^{-1}$,  there exist sufficiently small $\delta$ and large $\beta$ such that 
\begin{eqnarray}
\label{eqn:case_under_delta}
d_{1} \equiv \lambda\left(e^{\frac{\alpha}{2\sqrt{\beta}}+ \delta} - 1\right) - \mu_{\ell_{k}}(1 - e^{-\frac{\alpha}{\sqrt{k}}}) < 0.
\end{eqnarray}

We next consider the case $\frac{\alpha u_{\ell_{1}}}{r} \ge \delta$. Then, it implies that $u_{1} > 0$. 
From \eqn{Lyapunov condition1} and \eqn{inequality u 1}, 
\begin{eqnarray*}
&&e^{\frac{\alpha}{2r}} \left( \lambda e^{\frac{\alpha u_{\ell_{1}}}{r}} +  \sum_{i = j}^{k} \mu_{\ell_{\ell_{i}}} e^{-\frac{\alpha u_{\ell_{i}}}{r}} \right) - \left(\lambda + \sum_{i=j}^{k}\mu_{\ell_{i}} \right)\\
&& \qquad \qquad = (1 + e^{\frac{\alpha}{2r}} -1) \left( \lambda e^{\frac{\alpha u_{\ell_{1}}}{r}} +  \sum_{i = 1}^{k} \mu_{\ell_{\ell_{i}}} e^{-\frac{\alpha u_{\ell_{i}}}{r}} \right) - \left(\lambda + \sum_{i=1}^{k}\mu_{\ell_{i}} \right)\\
&& \qquad \qquad \le (1 + e^{\frac{\alpha}{2r}} -1) \left( \lambda e^{\frac{\alpha u_{\ell_{1}}}{r}} +  \sum_{i = 1}^{k} \mu_{\ell_{\ell_{i}}} e^{-\frac{\alpha u_{\ell_{1}}}{r}} \right) - \left(\lambda + \sum_{i=1}^{k}\mu_{\ell_{i}} \right)\\
&& \qquad \qquad = \left(e^{\frac{\alpha u_{\ell_{1}}}{r}} - 1 \right)\left(\lambda - \sum_{i=1}^{k} \mu_{\ell_{i}} e^{- \frac{\alpha u_{\ell_{1}}}{r}}\right) + (e^{\frac{\alpha}{2r}} -1)\left( \lambda e^{\frac{\alpha u_{\ell_{1}}}{r}} +  \sum_{i = 1}^{k} \mu_{\ell_{\ell_{i}}} e^{-\frac{\alpha u_{\ell_{1}}}{r}} \right)\\
&& \qquad \qquad \le \left(e^{\frac{\alpha u_{\ell_{1}}}{r}} - 1 \right)\left(\lambda - \sum_{i=1}^{k} \mu_{\ell_{i}} e^{- \frac{\alpha u_{\ell_{1}}}{r}}\right) + (e^{\frac{\alpha}{2r}} -1)\left( \lambda e^{\frac{\alpha}{\sqrt{k}}} +  \sum_{i = 1}^{k} \mu_{\ell_{\ell_{i}}} \right)\\
&& \qquad \qquad \le \left(e^{\delta} - 1 \right)\left(\lambda - \sum_{i=1}^{k} \mu_{\ell_{i}} e^{- \frac{\alpha}{\sqrt{k}}}\right) + (e^{\frac{\alpha}{2r}} -1)\left( \lambda e^{\frac{\alpha}{\sqrt{k}}} +  \sum_{i = 1}^{k} \mu_{\ell_{\ell_{i}}} \right),
\end{eqnarray*}
where the last inequality is given by $\frac{\alpha u_{1}}{r} \ge \delta$ and 
\begin{eqnarray*}
\lambda - \sum_{i=1}^{k} \mu_{\ell_{i}} e^{- \frac{\alpha u_{\ell_{1}}}{r}} \le  \lambda - \sum_{i=1}^{k} \mu_{\ell_{i}} e^{- \frac{\alpha}{\sqrt{k}}}<0,
\end{eqnarray*}
for $0 < \alpha < \sqrt{k} \log \rho^{-1}$.
Thus, for fixed $\alpha$ and $\delta$, we obtain \eqn{Lyapunov sufficient} for sufficiently large $\beta > 0$ satisfying 
\begin{eqnarray}
\label{eqn:case_over_delta}
 d_{2} \equiv \left(e^{\frac{\alpha}{2\sqrt{\beta}}} -1 \right)\left( \lambda e^{\frac{\alpha}{\sqrt{k}}} -  \sum_{i = 1}^{k} \mu_{\ell_{\ell_{i}}} \right) + \left(e^{\delta} - 1\right)\left(\lambda - \sum_{i=1}^{k} \mu_{\ell_{i}} e^{- \frac{\alpha}{\sqrt{k}}}\right) < 0.
\end{eqnarray}
We put $c_{1} = - \max(d_{1},d_{2})$. Then, from \eqn{case_under_delta} and \eqn{case_over_delta}, we have $c_{1} > 0$ and \eqn{Lyapunov sufficient}. 
This completes the proof since $\varphi(0,\theta \vc{1}) \le \mathbb{E}(e^{\theta  \sum_{i=1}^{k} L_{i}})$ for any $\theta \ge 0$.

\section{Proof of \lem{stationary inequality}}
\setnewcounter
\label{app:stationary inequality}

We prove \eqn{stationary inequality}. 
For this, we apply truncation argument for the moment generating functions $\varphi_{0U}$ and $\varphi_{+U}$. For each $n=1,2,\ldots,$ let 
\begin{eqnarray*}
g_{n}(x) = \min(x,n), \quad x \in \dd{R}.
\end{eqnarray*}
Then, for any $x,y \in \dd{R}$, it is easy to see that
\begin{eqnarray}
\label{eqn:min gn}
g_{n}(x+y) \le g_{n}(x) + \left\{
\begin{array}{ll}
y & x \le n,\\
0 & x > n.
\end{array}
\right.
\end{eqnarray}
The moment generating function $E(e^{g_{n}(\vc{\theta}\vc{Z})})$ is finite for each $n$ and any $\vc{\theta} \in \dd{R}^{k+1}$. So, we have, by the stationary equation \eqn{stationary 1},  
\begin{eqnarray*}
\mathbb{E}(e^{g_{n}(\vc{\theta}\vc{Z})}) = \sum_{U \in {\sr K} }\Big( \mathbb{E}\left(e^{g_{n}(\vc{\theta}\vc{Z} +\vc{\theta} \vc{X}^{(0U)})}1(\vc{Z} \in {\sr S}_{0U})\right) + \mathbb{E}\left(e^{g_{n}(\vc{\theta}\vc{Z} +\vc{\theta} \vc{X}^{(+U)})}1(\vc{Z} \in {\sr S}_{+U})\right) \Big).
\end{eqnarray*}
Since $\vc{X}^{(+U)}$ is independent for $\vc{Z}$, from \eqn{min gn}, we have, for any $U \in {\sr K}$,
\begin{eqnarray*}
\mathbb{E}\left(e^{g_{n}(\vc{\theta}\vc{Z} +\vc{\theta} \vc{X}^{(+U)})}1(\vc{Z} \in {\sr S}_{+U})\right) &\le& 
\mathbb{E}\left(e^{g_{n}(\vc{\theta}\vc{Z})+\vc{\theta} \vc{X}^{(+U)}}1(\vc{Z} \in {\sr S}_{+U}, \vc{\theta}\vc{Z} \le n)\right) \\
&&\quad + \mathbb{E} \left(e^{g_{n}(\vc{\theta}\vc{Z})}1(\vc{Z} \in {\sr S}_{+U}, \vc{\theta}\vc{Z} > n)\right)\\
&=& \gamma_{+U}(\vc{\theta})\mathbb{E}\left(e^{g_{n}(\vc{\theta}\vc{Z})}1(\vc{Z} \in {\sr S}_{+U}, \vc{\theta}\vc{Z} \le n )\right)\\ 
&&\quad  + \mathbb{E}\left(e^{g_{n}(\vc{\theta}\vc{Z})}1(\vc{Z} \in {\sr S}_{+U}, \vc{\theta}\vc{Z} > n)\right).
\end{eqnarray*}
We have similar result for $\vc{X}^{(0U)}$. By the decomposition of $\mathbb{E}(e^{g_{n}(\vc{\theta}\vc{Z})})$,
\begin{eqnarray*}
&&\sum_{U \in {\sr K}}(1 - \gamma_{+U}(\vc{\theta}))\mathbb{E}\left(e^{g_{n}(\vc{\theta}\vc{Z})}1(\vc{Z} \in {\sr S}_{+U}, \vc{\theta}\vc{Z} \le n)\right) \\
 &&\qquad \leq \sum_{U \in {\sr K}} (\gamma_{0U}(\vc{\theta}) - 1)\mathbb{E}\left(e^{g_{n}(\vc{\theta}\vc{Z})}1(\vc{Z} \in {\sr S}_{0U}, \vc{\theta}\vc{Z} \le n)\right).
\end{eqnarray*}
Hence, we obtain \eqn{stationary inequality} as $n \to \infty$ by the monotone convergence theorem.
We complete the proof since we can use a similar argument to \eqn{stationary inequality2}.

\section{Proof of \cor{JSQ rough marginal}}
\setnewcounter
\label{app:JSQ rough marginal}

As we already noted, we only need to prove that
\begin{eqnarray}
\label{eqn:rough upper bound}
  \limsup_{n \to \infty} \frac 1n \log \dd{P}( M = n ) \le \log \rho^{k}.
\end{eqnarray}
  From \lem{domain D}, we have, for $\eta_{0} < \log \rho^{-k}$ and $0 \le \eta_{1} < \frac 1{k-1} \log \rho^{-k}$,
\begin{eqnarray*}
  e^{\eta_{0}n} \dd{P}(M = n) \le \dd{E} (e^{\eta_{0} M}) \le \varphi(\eta_{0}, \eta_{1} \vc{1}) < \infty.
\end{eqnarray*}
Hence,
\begin{eqnarray*}
  \eta_{0}n + \log \dd{P}(M = n) \le \log \varphi(\eta_{0}, \eta_{1} \vc{1}),
\end{eqnarray*}
 which implies that
\begin{eqnarray*}
  \limsup_{n \to \infty} \frac 1n \log \dd{P}( M = n ) \le - \eta_{0}.
\end{eqnarray*}
  Thus, letting $\eta_{0} \uparrow \log \rho^{-k}$ yields \eqn{rough upper bound}.

\end{document}